\documentclass[
	12pt,
	]{preprint}
\usepackage{amssymb}
\usepackage{amsmath}
\usepackage{hyperref}
\usepackage{mhenvs}
\usepackage{mhequ}
\usepackage{mhsymb}
\usepackage{booktabs}
\usepackage{tikz}
\usepackage{mathrsfs}
\usepackage{times}
\usepackage{microtype}
\usepackage{comment}
\usepackage{wasysym}
\usepackage{centernot}
\usepackage[shortlabels]{enumitem}
\usepackage{breakurl}
\usepackage{makecell}
\usepackage{tikz-cd}

\usetikzlibrary{arrows.meta}

\usepackage{graphicx}

\makeatletter
\DeclareRobustCommand{\cev}[1]{%
  {\mathpalette\do@cev{#1}}%
}
\newcommand{\do@cev}[2]{%
  \vbox{\offinterlineskip
    \sbox\z@{$\m@th#1 x$}%
    \ialign{##\cr
      \hidewidth\reflectbox{$\m@th#1\vec{}\mkern4mu$}\hidewidth\cr
      \noalign{\kern-\ht\z@}
      $\m@th#1#2$\cr
    }%
  }%
}
\makeatother

\usetikzlibrary{shapes.misc}
\usetikzlibrary{shapes.symbols}
\usetikzlibrary{shapes.geometric}
\usetikzlibrary{decorations}
\usetikzlibrary{decorations.markings}

\colorlet{symbols}{black}
\definecolor{connection}{rgb}{0.7,0.1,0.1}
\definecolor{purple}{rgb}{0.6,0.3,1}

\def\sp{\mathord{s\kern-0.13em p}}

\tikzset{
	noise/.style={circle,draw = black, line width = 0.1mm, fill=green!80!black, inner sep=1pt, minimum size=1mm},
    bluedot/.style={circle, fill=blue!80, inner sep=1pt, minimum size=1mm},
	root/.style={circle,draw=black,line width=0.1mm,fill=red!80,inner sep=1pt, minimum size=1mm},
	rootp/.style={circle,draw=black,line width=0.1mm,fill=purple,inner sep=1pt, minimum size=1mm},
	broot/.style={circle,draw=black,line width=0.1mm,inner sep=1pt, minimum size=1.3mm},
	innern/.style={circle,draw = black, fill=yellow,line width=0.1mm, inner
sep=1pt, minimum size=1.0mm},
    dot/.style={circle,draw=black,line width=0.1mm,fill=black,inner sep=1pt, minimum size=1mm},
    dotred/.style={circle,fill=black!50,inner sep=0pt, minimum size=2mm},
    var/.style={circle,fill=black!10,draw=black,inner sep=0pt, minimum size=3mm},
    kernel/.style={semithick,shorten >=2pt,shorten <=2pt},
    kernels/.style={snake=zigzag,shorten >=2pt,shorten <=2pt,segment amplitude=1pt,segment length=4pt,line before snake=2pt,line after snake=5pt,},
    rho/.style={densely dashed,semithick,shorten >=2pt,shorten <=2pt},
    testfcn/.style={dotted,semithick,shorten >=2pt,shorten <=2pt},
    renorm/.style={shape=circle,fill=white,inner sep=1pt},
    labl/.style={shape=rectangle,fill=white,inner sep=1pt},
	kernels2/.style={very thick,draw=connection,segment length=12pt},
	contract/.style={draw=green!70!black,segment length=2pt},
	not/.style={thin,circle,fill=symbols,draw=connection,fill=connection,inner sep=0pt,minimum size=0.5mm},
	>=stealth,
        }

\makeatletter
\def\DeclareSymbol#1#2#3{\expandafter\gdef\csname MH@symb@#1\endcsname{\tikz[baseline=#2,scale=0.15,draw=symbols,line join=round]{#3}}\expandafter\gdef\csname MH@symb@#1s\endcsname{\scalebox{0.7}{\tikz[baseline=#2,scale=0.15,draw=symbols,line join=round]{#3}}}}
\def\<#1>{\csname MH@symb@#1\endcsname}
\makeatother

\DeclareSymbol{IXi}{0}{\draw (0,0) node[not] {} -- (0,1.5) node[xic] {};}
\DeclareSymbol{IXi^2}{-1}{\draw(-1,1) node[xic] {} -- (0,0) node[not] {} -- (1,1) node[xic] {};}
\DeclareSymbol{IXi^3}{0}{\draw(-1,1) node[xic] {} -- (0,0) node[not] {} -- (1,1) node[xic] {};\draw (0,0) -- (0,1.5) node[xic] {} ;}
\DeclareSymbol{IXi21}{-1}{\draw (0,0) -- (1,1) node[xi] {}; \draw[kernels2] (-1,1) node[xi] {} -- (0,0) node[not] {} -- (0,1.4) node[xi] {};}

\DeclareSymbol{Y}{-.8}{\draw (0,0) node[root] {} -- (0,1.5) node[noise] {};}
\DeclareSymbol{Y^3}{0}{\draw(-1,1) node[dot] {} -- (0,0) node {} -- (1,1) node[dot] {};\draw (0,0) -- (0,1.5) node[dot] {};}
\DeclareSymbol{IY^3}{2.4}{\draw(-1,2) node[noise](a) {} -- (0,1.2) node[dot](m) {}; \draw (m) -- (1,2) node[noise] (b) {};\draw (m) -- (0,2.5) node[noise](c) {}; \draw (0,0) node[root] {}--(m){};}

\DeclareSymbol{YY}{0}{\draw (-1,1.5) node[dot] {} --(0,0) node {} --  (1,1.5) node[dot] {}; \draw[kernel] (-1,1.5) {} -- (1,1.5) {} ; }

\DeclareSymbol{u0}{0}{\node[noise] at (0,0.5) {};}
\DeclareSymbol{red}{0}{\node[broot,fill=red!80] at (0,.5) {};}
\DeclareSymbol{yellow}{0}{\node[broot,fill=yellow] at (0,.5) {};}
\DeclareSymbol{purple}{0}{\node[broot,fill=purple] at (0,.5) {};}
\DeclareSymbol{black}{0}{\node[broot,fill=black] at (0,.5) {};}
\DeclareSymbol{dot}{0}{\node[dot] at (0,.5) {};}

\DeclareSymbol{3}{0}{\draw (-1,1) node[dot] {} -- (0,0) -- (1,1) node[dot]{}; \draw (0,0) -- (0,1.2) node[dot] {};}

\DeclareSymbolFont{timesoperators}{T1}{ptm}{m}{n}
\SetSymbolFont{timesoperators}{bold}{T1}{ptm}{b}{n}

\makeatletter
\renewcommand{\operator@font}{\mathgroup\symtimesoperators}
\makeatother

\numberwithin{equation}{section}

\def\cff{\mathcal{F}}
\def\ctt{\mathcal{T}}
\def\coo{\mathcal{O}}

\def\lt{\left}
\def\rt{\right}
\def\tree{\ctt}
\def\loc{\mathrm{loc}}
\def\Cloc{{C_\loc(\R\times\R^d)}}

\DeclareMathOperator{\dist}{dist} 
\DeclareMathOperator{\sgn}{sgn} 

\newcommand*{\ud}{\mathrm{\,d}}

\newcommand{\EE}{\mathbf{E}}
\newcommand{\RR}{\mathbf{R}}
\newcommand{\NN}{\mathbf{N}}
\newcommand{\ZZ}{\mathbf{Z}}

\newcommand{\TT}{\mathbf{T}}

\newcommand{\PP}{\mathbf{P}}

\newcommand{\mG}{\mathcal{G}}

\newcommand{\mN}{\mathcal{N}}
\newcommand{\mO}{\mathcal{O}}
\newcommand{\mP}{\mathcal{P}}

\newcommand{\mM}{\mathcal{M}}
\newcommand{\mT}{\mathcal{T}}
\newcommand{\mI}{\mathcal{I}}

\newcommand{\mF}{\mathcal{F}}

\newcommand{\mX}{\mathcal{X}}

\newcommand{\mE}{\mathcal{E}}

\newcommand{\mV}{\mathcal{V}}

\newcommand{\mf}[1]{\mathfrak{#1}}

\renewcommand{\l}{\ell}
\newcommand{\ve}{\varepsilon}

\let\f\frac
\def\ie{\textit{i.e.}\ }

\newcommand{\myarcb}{\begin{tikzpicture}
\draw[line width=0.35mm] (180:0.1) arc (180:0:0.1) ;
\end{tikzpicture}}

\newcommand{\myarcg}{\begin{tikzpicture}
\draw[line width=0.35mm, draw=green!70!black] (180:0.1) arc (180:0:0.1) ;
\end{tikzpicture}}

\newcommand{\carrow}{
\begin{tikzpicture}[line width = 0.1mm,scale=1.2]
\coordinate (center) at (-0.09,0.3);
\node[broot,fill=red](arrowh) at (-.09,0.3) {};
\draw[fill=yellow] (center) + (0, 0.054) arc (90:270:0.054);
\node[broot,fill=red](r) at (.09,0.3) {};
\draw[white,-{To[length=0.6mm,width=0.9mm,black]}] (0,0.382) -- (-0.04,0.372);
\draw (r) to [in = 60,out = 120] (arrowh) ;
\end{tikzpicture}
}

\newcommand{\glue}{ \, \underset{\carrow}{\sqcup} \,}

\begin{document}

\title{The  Allen--Cahn equation with generic initial datum}
\author{Martin Hairer$^1$, Khoa L\^e$^2$, and Tommaso Rosati$^1$}
\institute{Imperial College, UK, \email{\{m.hairer,t.rosati\}@imperial.ac.uk}
\and TU Berlin, Germany, \email{le@math.tu-berlin.de}}

\maketitle
\begin{abstract}
We consider the Allen--Cahn equation $\partial_t u- \Delta u=u-u^3$ with
 a rapidly mixing Gaussian field as initial condition. We show that provided that the amplitude of the initial
 condition is not too large, the
equation generates fronts described by nodal sets of the Bargmann--Fock Gaussian
field, which then evolve according to mean curvature flow. \\[.4em]

\end{abstract}

\section{Introduction} 
\label{sec:introduction}

The aim of the present article is to prove that nodal sets of a smooth
Gaussian field, known as the Bargmann--Fock field,
arise naturally as a random initial condition to evolution by mean curvature
flow.

This problem is related to understanding the long-time behaviour of mean
curvature flow: for a sufficiently generic initial datum composed of
clusters with typical lengthscale of order $1$, one expects that at time  $ t \gg 1 $
the clusters have coarsened in such a way that the typical lengthscale is of order $ \sqrt{t} $. In fact, one
would even expect that upon rescaling by $ \sqrt{t} $, the clusters
become self-similar / stationary at large times \cite{bray}.

An understanding of such a behaviour remains currently beyond the reach of rigorous 
analysis although upper bounds on the coarsening rate can
be proven via deterministic arguments, see e.g. \cite{otto2002} for the
related Cahn--Hilliard dynamics.
On the other hand, lower bounds on the rate of coarsening can only be expected 
to hold for sufficiently
generic initial conditions, since for degenerate initial conditions one
may not see any coarsening at all. 

This motivates the question, addressed in
the present work, of what such generic initial condition should look like.
Notably, the correlation structure we obtain is
the same as the first order approximation of the longtime two-point correlation predicted
by Ohta, Jasnow and Kawasaki \cite{OhtaKawasaki}.

A natural way to construct random initial conditions to mean curvature flow is
to consider the fronts formed by the dynamics of the Allen--Cahn equation
with white noise initial data. Unfortunately this is not feasible, since the
scaling exponent $-\f d2$ of white noise on $ \RR^{d} $, with $d \geqslant 2$, 
lies below or at the critical exponent $-1$ (or $-\f23$ depending on what one really means by ``critical''
in this context) below which one does not expect any form of local well-posedness result 
for the equation.

	Instead, we consider the following setting. Let $ \eta $ be a white
noise on $\R^d$ with $d \ge 2$ and let $\phi$ be a Schwartz test function 
	integrating to $1$. Fix an exponent $\alpha \in
(0,1)$, and
for each $\varepsilon
\in (0, 1)$ and $x\in\R^d$ define 
\begin{equ}[e:defEtax]
\phi_x^\varepsilon(y)=\varepsilon^{-d}
\phi(\varepsilon^{-1}(x-y)), \qquad \eta_\varepsilon(x)=\varepsilon^{\frac d2-
\alpha}\eta(\phi^\varepsilon_x)\;.
\end{equ}
Here, the exponents are chosen in such a way that $\phi_x^\eps$ converges to a Dirac distribution
and typical values of $\eta_\eps$ are of order $\eps^{-\alpha}$.
Our aim is to describe the limit as $ \ve \to 0 $ of the solutions to the
Allen--Cahn equation with initial datum $ \eta_{\ve} $
\begin{equation}\label{eqn:uep}
		(\partial_t- \Delta) u_\varepsilon= u_\varepsilon-u^3_\varepsilon\,,\quad u_\varepsilon(0,x)=\eta_\varepsilon(x)\;.
\end{equation}
The reason for restricting ourselves to $\alpha < 1$ is that $1$ is precisely the critical exponent
for which $\Delta \eta_\eps$ and $\eta_\eps^3$ are both of the same order of magnitude, i.e.\ when
$\alpha+2 = 3\alpha$.

For a fixed terminal time $t$, 
the behaviour of \eqref{eqn:uep} is not very interesting since one has $\eta_\varepsilon \to 0$ weakly, so one would expect to also have
$u_\eps\to 0$. This is trivially the case for $\alpha < \f23$ since one has $\eta_\eps \to 0$ in $\CC^\beta$ for every
$\beta < -\alpha$ and the Allen--Cahn equation is well-posed for arbitrary initial data in $\CC^\beta$ 
if (and actually only if) $\beta > -\f23$.
As a consequence of Proposition~\ref{prop:uuN} below, we will see that it \textit{is} still the case 
that $u_\eps(t)\to 0$ at fixed $t$ for any exponent
$\alpha < 1$, but this is a consequence of more subtle stochastic cancellations.

It turns out that the natural time and length scales at which one does see non-trivial behaviour are given by
\begin{equ}
		T_\varepsilon= \Bigl({d\over 2} - \alpha\Bigr) \log \varepsilon^{-1}\;,\qquad
		L_\varepsilon=\sqrt{T_\eps}\,.
\end{equ}
The main result of this article is that for $ \sigma > 1 $, $ u ( \sigma T_{\ve}, x L_{\ve}) \to
\pm 1 $ on sets $ \Omega_{\sigma}^{\pm 1} $, which evolve under mean curvature
flow.
The initial data (at time $\sigma = 1$) for that flow is given by the nodal domains of a centred
Gaussian field $\{\Psi_{1}(x):x\in\R^d \}$ with Gaussian correlation function, also
known as the Bargmann--Fock field:
\begin{equation}\label{eqn:def-Psi}
		\E[ \Psi_{\sigma}(x)\Psi_{\sigma}(y)]= \sigma^{- \frac{d}{2}
}\exp\lt(-\frac{|x-y|^2}{8 \sigma} \rt), \qquad \sigma > 0\,.
	\end{equation}
In what follows we will write $ \Psi $ short for $ \Psi_{1}.$
This Gaussian field emerges from the linearisation near zero of
\eqref{eqn:uep}, within a time interval of order $1$ (in the original variables) around 
$t_{\star} (\ve) $ with
\begin{equ}[eqn:t-star]
t_{\star} (\ve) \eqdef T_\eps + \tau_\star(\ve)\;,\quad \tau_\star(\ve) \eqdef \f d4 \log \log \eps^{-1} + \mf{c}, \quad \mf{c}
\eqdef \frac{d}{4} \log{ \big( 4 \pi (d - 2 \alpha) \big)}\;,
\end{equ}
where $ \mf{c} $ is chosen to make certain expressions look nicer.
In fact, roughly up to time $t_{\star} (\ve) $ the
dynamics of $ u_{\ve} $ is governed by the linearisation of
\eqref{eqn:uep} around the origin, but at that time the solution becomes of order $ 1
$ and starts to follow the flow $ \Phi \colon \RR \times \RR \to
[-1, 1] $ of the nonlinear ODE
	\begin{equ}\label{eqn:Phi-new}
		\partial_t \Phi=\Phi-\Phi^3, \qquad \lim_{t \to -\infty}
e^{-t}\Phi(t,u) = u,
	\end{equ}
which is given by
\begin{equ}
	\Phi(t,u)=\frac{u}{\sqrt{e^{-2t} + u^2}}\;.
\end{equ}
We will show that a good approximation of $ u_{\ve}(t_{\star}(\ve), x
L_{\ve})$ is given by $ \Phi(0, \Psi(x)) $.

Then, after an additional time of order $ T_{\ve} $, $ u_{\ve} $ is finally governed by the large
scale behaviour of the
Allen--Cahn equation with random initial condition $ \Phi (0, \Psi) $, namely
by the mean curvature flow evolution of the set $ \{ \Psi =0 \} $.
To capture the different time scales we introduce the process
\begin{equ}\label{eqn:capital-U}
U_{\ve}(\sigma, x) \eqdef 
(e^{(1-\sigma) T_{\ve}}\vee 1)
u_{\ve}( \sigma T_{\ve} + \tau_{\star}(\ve), x L_{\ve})\;,
\end{equ}
and we summarise our results in the theorem below.
\begin{theorem}\label{thm:intro}
The process $ U_{\ve} $ converges in law as follows for $ \ve \to 0 $
\begin{equs}
U_{\ve}(\sigma, x) & \to \Psi_{\sigma}(x), \quad  & \sigma \in (0, 1)&, x \in
\RR^{d}\;, \\
U_{\ve}( 1 + t/T_\eps, x ) & \to \Phi( t , \Psi(x)), \quad
& t \in \RR&, x \in \RR^{d}\;.
\end{equs}
The convergences hold for fixed $t$ or $\sigma$ and locally uniformly over the variable $x$.

Denote furthermore by $ \Omega^{+}_{\sigma}, \Omega^{-}_{\sigma}$ respectively the mean
curvature flow evolutions of the sets $ \Omega^{+}_{1} = \{ x  \ \colon \
\Psi(x) > 0 \}, \ \Omega^{-}_{1} = \{ x  \ \colon \ \Psi(x) < 0 \}$. Then,
there exist a coupling between $\{U_\ve\}_{\ve \in (0, 1)}$ and $\Psi$ such that, for every $\sigma > 1$,
one has 
$\lim_{\eps \to 0} U_{\ve}(\sigma, x) = \pm 1$ locally uniformly for $x \in \Omega^{\pm}_{\sigma}$, in probability.
\end{theorem}

\begin{remark}\label{rem:fattening}
Although level set solutions to mean curvature flow with continuous initial
datum are unique, they might exhibit fattening, meaning that the evolution of the zero level set $
\Gamma_{\sigma} = \RR^d \setminus (\Omega^{+}_{\sigma} \cup \Omega^{-}_{\sigma}) $ might have a non-empty
interior for some $\sigma > 1$. In dimension $ d=2 $ this does not happen, as the initial interface is
composed of disjoint smooth curves which evolve under classical mean curvature
flow, c.f.\ Corollary~\ref{cor:2D-mcf}. In dimension $ d \geqslant 3 $ this
picture is less clear, and in addition the initial nodal set was recently
proven to contain unbounded
connected components \cite{duminil}.
\end{remark}

\subsection{Structure of the paper}

The rest of this article is divided as follows. In Section~\ref{sec:main-results}, we
reformulate Theorem~\ref{thm:intro} more precisely and we provide its proof, taking for granted the
more technical results given in subsequent sections. 
In
Section~\ref{sec:wild-expansion} we study a Wild expansion of the solution to
\eqref{eqn:uep} which allows us to take advantage of stochastic cancellations to control the
solutions on a relatively short initial time interval of length smaller than $(1 - \alpha)
\log{\ve^{-1}}$. 
In Section~\ref{sec:weak_convergence} we then show how the
Bargmann--Fock field arises from this expansion and we track solutions up to timescales of order $t_\star$. 
In Section~\ref{sec:convergence-mcf}, we finally conclude by showing that our estimates are sufficiently
tight to allows us to patch into existing results on the convergence to mean curvature flow on
large scales.

\subsection{Notations}
For a set $ \mathcal{X} $ and two functions $ f, g \colon \mX
\to \RR $, we write $ f \lesssim g $ if there exists a constant $ c > 0 $ such
that $ f(x) \leqslant c g(x) $ for all $ x \in \mX $. If $ \mX $ is a locally
compact metric space we define $C_\loc(\mX) $ as the space of continuous and
locally bounded functions, equipped with the topology of uniform convergence on
compact sets. We will write $ P_{t},
P^{1}_{t} $ for the following semigroups on $\RR^{d}$:
\begin{equ}
P_{t} \varphi = e^{t \Delta} \varphi, \qquad P^{1}_{t} \varphi = e^{t(\Delta +
1)} \varphi, \quad t \geqslant 0.
\end{equ}
We denote with $ \sgn \colon \RR \to \{ -1, 0, 1 \} $ the sign function, with
the convention that $ \sgn(0) = 0$.

\subsection*{Acknowledgements}

{\small
MH gratefully acknowledges support from the Royal Society through a research
professorship, grant number RP\textbackslash R1\textbackslash 191065.
}

\section{Main results}\label{sec:main-results}

In what follows we consider $ (t, x) \mapsto u_{\ve} (t, x) $ as a process on $ \RR \times
\RR^{d} $ by defining $ u_{\ve}(t, x) = u_{\ve}(0, x) $ for $ t < 0
$. We then show the first part of Theorem~\ref{thm:intro}, namely
\begin{theorem}\label{thm:convergvence-Gaussian}
Let $ u_{\ve} $ solve \eqref{eqn:uep} and $\Psi_{\sigma}, \Phi, U_{\ve}$ be
given respectively by \eqref{eqn:def-Psi}, \eqref{eqn:Phi-new} and
\eqref{eqn:capital-U}. Then
\begin{enumerate}
\item For any $ \sigma \in (0, 1) $ the process $ \{ U_{\ve} (\sigma, x)  \ \colon \
x \in \RR^{d} \} $ converges locally uniformly in law to $ \{ \Psi_{\sigma}
(x)  \ \colon \ x \in \RR^{d} \} $.
\item  The process $\{u_\varepsilon(t_\star(\varepsilon)+t, x
L_{\ve}):(t,x) \in \R \times \R^d \}$ converges locally uniformly in law to
$\lt\{\Phi\left(t,
\Psi(x) \right):(t,x) \in \R \times \R^d \rt\}$.
\end{enumerate}
\end{theorem}
\begin{proof}

The first statement follows, for $\sigma \in \big(0, \frac{1 -\alpha}{ d/2 -
\alpha } \big)$, from similar calculations as in Lemma~\ref{lem:Pu1convg} in
combination with Proposition~\ref{prop:uuN}. For $\sigma \in \big[\frac{1 -\alpha}{ d/2 -
\alpha }, 1 \big) $ it follows again by similar calculations as in
Lemma~\ref{lem:Pu1convg}, this time in combination with Lemma~\ref{lem:S}.
Let us pass to the second statement. For some $ \overline{\alpha} \in (\alpha, 1) $ consider 
\[ t_{1}(\ve) = ( \overline{\alpha} - \alpha) \log{ \ve^{-1}}\;, \qquad
t_{2}(\ve) = T_\eps -
\frac{1}{2} \log{ \log{ \ve^{-1}}}\;.\]
		It is shown in Proposition~\ref{prop:d3} that the limit as
$\eps \to 0$ of the process 
$(t,x)\mapsto u_\varepsilon(t_\star+t,L_\varepsilon x) $ is identical, in
probability, to that of $(t,x)\mapsto w^N_\varepsilon(t_\star+t,L_\varepsilon x)$ for some fixed $N$ sufficiently large.
		From the definition of $ w^{N}_{\ve} $, we have
		\begin{equation*}
			w^N_\varepsilon(t,x) = 
			\frac{P^1_{t-t_1}u^N_\varepsilon(t_1,x)}{\(1+(1-e^{-2(t-t_2)})(P^1_{t-t_1}u^N_\varepsilon(t_1,x))^2 \)^{1/2}}\;,
		\end{equation*}
where $ u^{N}_{\ve} $ is the Wild expansion truncated at level $N$ given by \eqref{def:uN}.
		Applying Lemma~\ref{lem:Pu1convg}, we see that the process $(t,x)\mapsto P^1_{t_\star+t-t_1}u^N_\varepsilon(t_1,L_\varepsilon x) $ converges to $(t,x)\mapsto e^t \Psi(x)$ in $\Cloc$ in distribution. 
		Since furthermore $t_\star \gg t_1$, it follows that the process $(t,x)\mapsto  w^N_\varepsilon(t_\star+t,L_\varepsilon x)$ converges in $\Cloc$ in distribution to the process
		\begin{equation*}
			(t,x)\mapsto \frac{ e^{t}\Psi(x)}{\(1+ e^{2 t}\Psi(x)^2 \)^{1/2}} = 
			\Phi(t, \Psi(x))\;,
		\end{equation*}
thus concluding the proof.
	\end{proof}
We now turn to the proof of the second part of Theorem~\ref{thm:intro}. This is trickier to formulate
due to the potential fattening of level sets already alluded to in
Remark~\ref{rem:fattening}. 
In particular, let us denote by $ (\sigma, x) \mapsto v(f ; \sigma, x) \in
\{ -1, 0, 1 \}$ the sign of the level set solution to mean curvature flow
associated to the initial interface $
\{ f = 0 \} $ in the sense of Definition~\ref{def:mcf}, with the
difference that the initial condition is given at time $ \sigma =1 $ in the
present scale. We will then
write $ \Gamma $ for the random interface $ \Gamma = \{ (\sigma, x) \in [1, \infty)
\times\RR^{d}
\ \colon \ v(\Psi ; \sigma, x) = 0 \} \subseteq \RR^{d+1} $. In addition, in
order to define locally uniform convergence in the complement of $ \Gamma $, 
let us introduce for any $ \delta \in (0, 1) $ the random sets
\begin{equs}[eqn:K-sets]
K_{\delta} & = \{ z = (\sigma, x) \in (1, \infty) \times \RR^{d}  \, \colon  \, | z | \leqslant
\delta^{-1}, \ \sigma > 1 + \delta,  \ d(z, \Gamma) \geqslant \delta\}\\
K_{\delta}^{1} & = \{ x \in \RR^{d}  \ \colon \ | x | \leqslant \delta^{-1},
\ d(x, \Gamma_{1}) \leqslant \delta \}\;.
\end{equs}
Here $ d(p, X) $ is the usual point-to-set distance. Furthermore, we can
define the (random) norms
\begin{equs}
\| f \|_{K_{\delta}} = \sup_{(\sigma, x) \in K_{\delta}} | f(\sigma, x) |
\end{equs}
and analogously for $ K_{\delta}^{1} $.
With these definitions at hand, the next
result describes the formation of the initial interface, which appears if we
wait and additional time of order $ \log{ \log{ \ve^{-1}}} $. Hence we define,
for $ \kappa > 0$:
\begin{equ}[eqn:t-star-kappa]
t^\kappa_\star(\varepsilon) \eqdef t_\star(\ve) +\kappa \log\log \varepsilon^{-1} \,.
\end{equ}

	\begin{proposition}\label{prop:front-formation}
Consider any sequence of times $ \{ t(\ve) \}_{\ve \in (0, 1)} \subseteq (0,
\infty) $ such that for some $ \kappa >0 $
\begin{equs}
\liminf_{\ve \to 0} \{ t (\ve) - t^{\kappa}_{\star}(\ve) \} \geqslant 0 \;, \qquad
\limsup_{\ve \to 0} \frac{t(\ve) - t_{\star}(\ve)}{T_{\ve}} = 0\;.
\end{equs}
Then there exists a coupling between $ u_{\ve} $ and
$ \Psi $ such that for any $ \delta, \zeta \in (0, 1) $ 
\begin{equs}
\lim_{\ve \to 0} \PP \( \| u_\varepsilon(t(\varepsilon), \cdot L_{\ve}) -
\sgn(\Psi( \cdot )) \|_{K^{1}_{\delta}} > \zeta \) = 0\;.
\end{equs}
	\end{proposition}

\begin{proof}
Up to taking subsequences, it suffices to prove the result in either of the
following two cases:
\begin{equs}
\liminf_{\ve \to 0} \{ t(\ve) - t^{\frac{1}{2}}_{\star}(\ve) \} >0\;, \qquad \text{
or } \qquad \limsup_{\ve \to 0} \{ t(\ve) - t^{\frac{1}{2}}_{\star}(\ve) \}
\leqslant 0\,.
\end{equs}
In the first case, the result is a consequence of
Proposition~\ref{prop:slower-than-mcf} while in the second case it follows 
from Lemma~\ref{lem:front-formation-loglog}. In both cases the choice of the
coupling is provided by Lemma~\ref{lem:prob-space}.
\end{proof}
In the case $ t(\ve) =
t_{\star}^{\kappa}(\ve)  $ for some $ \kappa \in (0, \frac{1}{4} ) $
one can also obtain the above result as a consequence of Proposition~\ref{prop:d3}.
Finally, we show that if we wait for an additional time of order $ T_{\ve} $, the interface moves under mean
curvature flow. This result is a consequence of Proposition~\ref{prop:conv-mcf}. 
 
\begin{theorem}\label{thm:mcf}
Consider $ U_{\ve}(\sigma, x)$ as in \eqref{eqn:capital-U} for $
\sigma > 1, x \in \RR^{d} $. There exists a coupling between $ u_{\ve} $ and
$ \Psi $ such that for any $ \delta, \zeta \in (0, 1) $ 
\begin{equs}
\lim_{\ve \to 0} \PP \big( \| U_{\ve}(\cdot) - v (\Psi; \cdot) \|_{K_{\delta}} >
\zeta \big) = 0\,.
\end{equs}
\end{theorem}


\section{Wild expansion} 
\label{sec:wild-expansion}

The next two sections are devoted to tracking $ u_{\ve} $ for times up to and around
$ t_{\star} (\ve) $. In order to complete this study we divide the interval
$ [0, t_{\star}(\ve)] $ into different regions. Let $\bar \alpha\in(\alpha,1)$ be fixed and define
	\begin{equs}[eqn:t-1-2]
		t_1(\varepsilon)& \eqdef (\bar \alpha- \alpha)\log \varepsilon^{-1}, \\
		t_2(\varepsilon)& \eqdef \(\frac d2- \alpha\)\log \varepsilon^{-1}-\frac12\log\log \varepsilon^{-1}\,.
	\end{equs}
We observe that for $ \ve $ small one has
\begin{equ}
0 \ll t_{1}(\ve) \ll t_{2}(\ve) \ll t_{\star} (\ve)\;.
\end{equ}
The specific forms of these times are chosen such that in complement to the moment estimates, the error terms in various estimates which appear below remain small when $\varepsilon\to0$. 
	The constant $-1/2$ which appears in $t_2(\varepsilon)$ can be replaced by any negative number, and in spatial dimensions 3 or higher can be replaced by $0$.
	The separation of time scales is due to that fact that the linear and nonlinear terms in \eqref{eqn:uep} have different effects on the solution over each period.

	During the time period $[0,t_1(\ve)]$, the Laplacian dominates and \eqref{eqn:uep} can be treated as a
perturbed heat equation. One expects that a truncated Wild expansion $
u_{\ve}^{N} $ (where $ N $ is the level of the truncation) provides a
good approximation for $u_\varepsilon$ during this initial time period, see
Proposition \ref{prop:uuN} at the end of this section.
During the time period $[t_1(\ve),t_2(\ve)]$, the linear term increases
the size of the solution from $\CO(\eps^{\theta})$ for 
any $ 0 < \theta < \frac{d}{2} - \overline{\alpha} $ to
almost a size of order 1. During this
period, the estimate \eqref{est:uuN} in Proposition~\ref{prop:uuN} does not reflect
the actual size of the solution $u_\varepsilon$. However, the leading order
term $X^\bullet_{\ve}$ in $u^N_{\ve}$ (which is the solution to the linearisation near $ 0
$ of \eqref{eqn:uep}) at time $t_2(\ve)$ is
of order $(\log \varepsilon^{-1}) ^{-d/4}$. Hence the
solution $u_\varepsilon$ still remains small, and we expect that the non-linear
term $u_\varepsilon^3$ in \eqref{eqn:uep} is negligible: see Lemma~\ref{lem:S}.
Eventually one starts to see the classical dynamic of the Allen--Cahn equation
during $[t_2(\ve),t_\star(\ve)]$, as explained in Proposition~\ref{prop:d3}. 

To establish the aforementioned results, we will frequently make use of the following 
formulation of the maximum principle. 
	\begin{lemma}\label{lem:abs.v}
		Let $(a,b) \subset \R$ be a finite interval and let $G,R$ be measurable functions on $(a,b)\times\R^d$ such that $G$ is non-negative. Suppose that $v$ is a function in $C^2([a,b]\times\R^d) $ satisfying
		\begin{equation*}
			(\partial_t- \Delta-1)v=-Gv+R \quad\textrm{on}\quad (a,b)\times\R^d
			\,.
		\end{equation*}
		Then,
		\begin{equation*}
			|v(t,x)|\le P^1_{t-a}|v|(a,\cdot)+ \int_{a}^{t}
P^1_{t - s} |R|(s) \ud s \quad\forall (t,x)\in [a,b]\times\R^d\,.
		\end{equation*}
	\end{lemma}

\subsection{Moment estimates} 
\label{sec:moment_estimates}
	Let $\tree_3$ be the set of rooted ternary ordered trees defined inductively by postulating that 
	$\bullet \in \tree_3$ and, for any $\tau_i\in \tree_3$, one has $\tau=[\tau_1,\tau_2,\tau_3] \in \CT_3$. In general, for a finite collection of trees $\tau_1,\dots, \tau_n$ we denote by $\tau = [\tau_1,\dots,\tau_n]$ the tree obtained by connecting the roots of the $\tau_i$ with a 
	new node, which in turn becomes the root of the tree $\tau$:
	\begin{equation*}
		[\tau_1,\dots,\tau_n]=
		\begin{tikzpicture}[baseline=2]
			\node[dot] (0) at (0,0) {};
			\draw (0) to (-.5,0.4) node at (-.5,.5) {$\tau_1$} ;
			\draw (0) to (-.2,0.4) node at (-.2,.5) {$\tau_2$} ;
			\draw (0) to (0,0.4) node at (0.2,.5) {\tiny$\cdots$} ;
			\draw (0) to (0.2,0.4) ;
			\draw (0) to (0.5,0.4) node at (0.5,.5) {$\tau_n$} ;
		\end{tikzpicture}\,.
	\end{equation*}
These trees are \emph{ordered} in the sense that $
[\tau_{1}, \tau_{2}, \tau_{3}] \neq [\tau_{1}, \tau_{3}, \tau_{2}] $ unless
$ \tau_{2} = \tau_{3} $, and similarly for all other permutations. For example
we distinguish the following two trees:
\begin{equ}
	\begin{tikzpicture}[baseline=2]
		\node[dot](root) at (0,0) {};
		\node[dot](left) at (-0.21,.17) {};
		\node[dot](right) at (.21,.17) {};
		\node[dot](center) at (.0,.17) {};
		\node[dot](uleft) at (-0.19,.35) {};
		\node[dot](ucenter) at (.0,.35) {};
		\node[dot](uright) at (0.19,.35) {};
		\draw (root) to (left) {};
		\draw (root) to (right) {};
		\draw (root) to (center) {};
		\draw (center) to (uleft) {};
		\draw (center) to (ucenter) {};
		\draw (center) to (uright) {};
	\end{tikzpicture} \neq 
	\begin{tikzpicture}[baseline=2]
		\node[dot](root) at (0,0) {};
		\node[dot](left) at (-0.21,.17) {};
		\node[dot](right) at (.21,.17) {};
		\node[dot](center) at (.0,.17) {};
		\node[dot](uleft) at (0.02,.35) {};
		\node[dot](ucenter) at (.21,.35) {};
		\node[dot](uright) at (0.40,.35) {};
		\draw (root) to (left) {};
		\draw (root) to (right) {};
		\draw (root) to (center) {};
		\draw (right) to (uleft) {};
		\draw (right) to (ucenter) {};
		\draw (right) to (uright) {};
	\end{tikzpicture} \; .
\end{equ}
For $\tau \in \tree_3$, we write
$i(\tau)$ for the number of inner nodes of $[\tau]$ (\ie\ nodes which are not
leaves, but including the root) and $\ell(\tau)$ be the number of leaves of $[\tau]$. Then $\ell$ and $i$ are related by
	\begin{equation}
		\ell(\tau)=2i(\tau)+1\,.
	\end{equation}
	For each positive integer $n$, $\pi_n= (2n-1)!!$ is the number of pairings of $2n$ objects so that each object is paired with exactly one other object.
	For $\tau=[\tau_{1}, \tau_{2}, \tau_{3}]\in\tree_3$, we define $X^\tau_\varepsilon$ inductively by solving
	\begin{equs}\label{eqn:Xtrunk}
		\partial_t X^{\bullet}_{\ve}&=(\Delta+1)X^{\bullet}_{\ve}
\,,&\quad X^{\bullet}_{\ve}(0,\cdot)&= \eta_\varepsilon(\cdot)\,,\\
\label{eqn:Xtau}
		\partial_t X^\tau_\varepsilon&=(\Delta+1)X^\tau_\varepsilon-X^{\tau_1}_\varepsilon X^{\tau_2}_\varepsilon X^{\tau_3}_\varepsilon\,,&\quad X^\tau_\varepsilon(0,\cdot)&=0\;.
	\end{equs}
	For $N \ge 1$,
	we write $\tree_3^N \subset \tree_3$ for the trees
	with $i(\tau)\le N$ and we define 
	\begin{equation}\label{def:uN}
	 	u^N_\varepsilon \eqdef \sum_{\tau\in\tree_3^N}X^\tau_\varepsilon\,,
	\end{equation}
	which is a truncated Wild expansion of $ u_{\ve} $.
	Then $u^N_\varepsilon$ satisfies the following equation
	\begin{align}\label{eqn:uN}
		(\partial_t- \Delta-1)u^N_\varepsilon=-(u^N_\varepsilon)^3+R^N_\varepsilon\,,\quad u^N_\varepsilon(0,\cdot)=u_\varepsilon(0,\cdot)\,,
	\end{align}
	where
	\begin{equation}\label{eqn:RN}
		R^N_\varepsilon=\sum_{\tau_1,\tau_2,\tau_3\in\tree^N_3:\, [\tau_1,\tau_2,\tau_3]\notin\tree^N_3}X^{\tau_1}_\varepsilon X^{\tau_2}_\varepsilon X^{\tau_3}_\varepsilon \,.
	\end{equation}
	Indeed, using \eqref{eqn:Xtau} and \eqref{eqn:Xtrunk}, we have
	\begin{align*}
		(\partial_t- \Delta-1)u^N_\varepsilon
		&=-\sum_{\tau=[\tau_1,\tau_2,\tau_3]\in\tree_3^N\setminus\{\bullet \}}X^{\tau_1}_\varepsilon X^{\tau_2}_\varepsilon X^{\tau_3}_\varepsilon
		\\&=-\Big(\sum_{\tau\in\tree^N_3}X^\tau_\varepsilon \Big)^3+R^N_\varepsilon
		=-(u^N_\varepsilon)^3+R^N_\varepsilon\,.
	\end{align*}
	Let us now introduce some graphical notations to represent conveniently
$X^\tau_{\ve}$ and the integrals that will be associated to them.
	In what follows the negative heat kernel $-P^1(t-s,x-y)$ (here the negative 
sign appears to keep track of the sign of the polynomial term in
\eqref{eqn:Xtau}) will be represented by a directed edge
	\begin{equation*}
		- K_{\myarcb}
((t,x), (s, y))= -P^1(t-s,x-y)=
		\begin{tikzpicture}[baseline=10]
		\node[root,label=0:{$(t,x)$}](l) at (0,0) {};
		\node[root,label=0:{$(s,y)$}](h) at (0.3,1){};
		\draw[->] (h) to (l);
		\end{tikzpicture}\,.
	\end{equation*}
	Each endpoint of an edge represents a space-time variable and the kernel is evaluated at their difference. 
	Three different nodes \tikz[baseline=-2]\node[root] {}; ,
\tikz[baseline=-2]\node[dot] {}; and \tikz[baseline=-2]\node[noise] {}; can be
attached to an end of an edge, which represent respectively a fixed space-time
variable, integration with respect to Lebesgue measure $ \ud s \ud y$, and
integration with respect to the (random) measure $ - \delta(s) \eta_{\varepsilon}(y) \ud s
\ud y$. Here, $\delta$ is the Dirac mass at 0 and the minus sign is again
just a matter of convention. For example, we have
	\begin{equ}
		\begin{tikzpicture}[baseline=5]
		\node[root,label=0:{$(t,x)$}](l) at (0,0) {};
		\node[noise](h) at (0,0.6){};
		\draw[->] (h) to (l);
		\end{tikzpicture}
		=\int_{\R^d}P^1(t,x-y) \eta_{\ve} (y) \ud y=X^\bullet_{\ve}(t,x)\,,
	\end{equ}
as well as 
	$
		X^{\<3>}_\varepsilon = 
		\begin{tikzpicture}[baseline=5]
		\node[root](r) at (0,0) {};
		\node[dot](1) at (0,0.3) {};
		\node[noise](11) at (-.2,.6) {};
		\node[noise](12) at (0,.6) {};
		\node[noise](13) at (.2,.6) {};
		\draw[->] (11) to (1);
		\draw[->] (12) to (1);
		\draw[->] (13) to (1);
		\draw[->] (1) to (r);
		\end{tikzpicture}
	$.
	
	Since all the edges follow the natural direction toward the root, we
may drop the arrows altogether and just write $X^\bullet_{\ve}=\<Y>$ and
$X^{\<3>}_{\ve}=\<IY^3>$.
For a general tree $\tau \in \tree_3$,
$X^{\tau}_\varepsilon$ is represented by the tree $[\tau]$ with its root
coloured red and leaves coloured green.

	Given a tree $\tau\in\tree_3$, we would like to estimate the moment $\E [X^\tau_\varepsilon(t,x)]^2$.
	The contracting kernel (see \eqref{e:defEtax})
	\begin{equation}\label{eqn:contraction-kernel}
	 	K_{\myarcg} ( (s,y), (\bar s,\bar y)) = \varepsilon^{d-2
\alpha}\delta(s)\delta(\bar s)\int_{\R^d} \phi^\varepsilon_y(z)
\phi^\varepsilon_{\bar y}(z) \ud z\,,
	\end{equation}
	which appears when one contracts two noise nodes (\ie two green nodes),
is represented as a green edge. Note that since these kernels are symmetric,
green edges are naturally undirected. Moreover, since we are
only interested in upper bound and up to replacing $ \varphi^{\ve} $ with $ |
\varphi^{\ve} | $, we may assume that \[ \varphi^{\ve} (x) \geqslant 0, \qquad \forall x \in
\RR^{d}.  \] Then for each $\varepsilon$, we define
the \emph{positive} kernel
$$p_t^\varepsilon(x)=\int_{\R^d}p_t(x-y)\phi^\varepsilon_y(0) \ud y, \qquad
p_{t}(x) = \frac{e^{ - \frac{| x |^{2}}{ 4t} }}{(4 \pi t )^{\frac{d}{2}}}.$$
From Young's inequality, we see that
	\begin{align*}
		\|p^\varepsilon_t\|_{L^2(\R^d)}\le (\|p_t\|_{L^2(\R^d)}\|\phi^\varepsilon_\cdot\|_{L^1(\R^d)})\wedge(\|p_t\|_{L^1(\R^d)}\|\phi^\varepsilon_\cdot\|_{L^2(\R^d)})
		\le c (t\vee \varepsilon^2)^{-\frac d4}\;,
	\end{align*}
	for some constant $c>0$.
	Another elementary observation is the bound
	\begin{equs}
		\begin{tikzpicture}[baseline=-10,scale=0.5]
			\node [dot](b) at (0,0) {};
			\node [dot](c) at (1,0) {};
			\node [root,label=180:{$(t_1,x_1)$}](a)  at (-1,-1) {};
			\node [root,label=0:{$(t_2,x_2)$}](d) at (2,-1) {};
			\draw[->] (b) to  (a) {};
			\draw[contract] (b) to (c) {};
			\draw[->] (c) to (d) {};
		\end{tikzpicture}
		&=e^{t_1+t_2} \varepsilon^{d-2
\alpha}\int_{\R^d}p^\varepsilon_{t_1}(x_1-z)p^\varepsilon_{t_2}(x_2-z) \ud z\qquad
		\\&\le (c \varepsilon^{\frac d2- \alpha})^2 e^{t_1+t_2} (t_1\vee \varepsilon^2)^{-\frac d4} (t_2\vee \varepsilon^2)^{-\frac d4} \,.\label{e:usefulBound}
	\end{equs}
	Here we did not use any green nodes
because the distinction between green and black nodes is captured
by the difference between the kernels $ K_{\myarcb} $ and $ K_{\myarcg} $, and
as per assumption every black node indicates integration over both space and
time variables.

	We begin with an estimate for $\E\<Y>^2(t,x)$. The second moment is
obtained by contracting two leaves,
	\begin{equation*}
		\E[\<Y>^2(t,x)]=
		\begin{tikzpicture}[baseline=5]
			\node[root,label=below:{\tiny$(t,x)$}](root) at (0,0) {};
			\node[dot](left) at (-0.5,.6) {};
			\node[dot](right) at (.5,.6) {};
			\draw (root) to (left) {};
			\draw (root) to (right) {};
			\draw[contract] (right) [bend right=30] to (left) {};
		\end{tikzpicture}
		\le(ce^t \varepsilon^{\frac d2- \alpha})^2(t\vee \varepsilon^2)^{-\frac d2}  \,.
	\end{equation*}
	For a general tree $\tau$, the second moment $\E |X^\tau_{\ve}(t,x)|^2$ is obtained by summing 
	over all pairwise contractions between the leaves of $[\tau,\bar\tau]$, where $\bar \tau$ is an identical copy of the tree $\tau$. For instance,
	\begin{equation}\label{eqn:32nd}
		\E\<IY^3>^2=
		6
		\begin{tikzpicture}[baseline=5]
			\node[root](root) at (0,0) {};
			\node[dot](l) at (-0.5,0.3) {};
			\node[dot](r) at (.5,.3) {};
			\node[dot](r1) at (.7,.6) {};
			\node[dot](r2) at (.5,.6) {};
			\node[dot](r3) at (.3,.6) {};
			\node[dot](l1) at (-0.7,0.6) {};
			\node[dot](l2) at (-0.5,0.6) {};
			\node[dot](l3) at (-0.3,0.6) {};
			\draw (l) to (root) to (r) {};
			\draw (l) to (l1) {};
			\draw (l) to (l2) {};
			\draw (l) to (l3) {};
			\draw (r) to (r1) {};
			\draw (r) to (r2) {};
			\draw (r) to (r3) {};
			\draw[contract] [bend left=40] (l1) to (r1) {};
			\draw[contract] [bend left=30] (l2) to (r2) {};
			\draw[contract] [bend left=10] (l3) to (r3) {};
		\end{tikzpicture}
		+9
		\begin{tikzpicture}[baseline=5]
			\node[root](root) at (0,0) {};
			\node[dot](l) at (-0.5,0.3) {};
			\node[dot](r) at (.5,.3) {};
			\node[dot](r1) at (.7,.6) {};
			\node[dot](r2) at (.5,.6) {};
			\node[dot](r3) at (.3,.6) {};
			\node[dot](l1) at (-0.7,0.6) {};
			\node[dot](l2) at (-0.5,0.6) {};
			\node[dot](l3) at (-0.3,0.6) {};
			\draw (l) to (root) to (r) {};
			\draw (l) to (l1) {};
			\draw (l) to (l2) {};
			\draw (l) to (l3) {};
			\draw (r) to (r1) {};
			\draw (r) to (r2) {};
			\draw (r) to (r3) {};
			\draw[contract] [bend left=50] (l1) to (l2) {};
			\draw[contract] [bend right=50] (r1) to (r2) {};
			\draw[contract] [bend left=50] (l3) to (r3) {};
		\end{tikzpicture}\,.
	\end{equation}
	Let us take a closer look at each term. In particular, we can extend
our graphical notation in order to obtain efficient estimates for such trees.


\textbf{More tree spaces.} First, we observe that each of the graphs above comes with a
specific structure: it consists of a tree with an even number of leaves,
together with a pairing among the leaves. In addition the root is coloured
red and every node that is not a leaf is the root of up to three planted trees
in $ [\mT_{3}] $. This suggest the introduction of the following space of
paired forests $ \mF $ (later on it will be convenient to work with forests and
not only trees).

Let $ \mT $ be the set of all finite rooted 
trees and for $ \{ \tau_{i} \}_{i=1}^{k} \subseteq \mT $ define $ \l(\tau_{1} \sqcup \cdots \sqcup \tau_{k} ) =
\l(\tau_{1}) + \cdots \l(\tau_{k}) $ the total number of leaves of a forest.
Then define $ \mF $ as
\begin{equs}
\mF = \Big\{ ( \tau_{1}\sqcup \cdots \sqcup \tau_{k}, \gamma) & \ \colon \ k \in
\NN, \ \tau_{i} \in \mT , \\ 
& \text{ with } \l(\tau_{1}\sqcup \cdots \sqcup
\tau_{k}) \in 2 \NN  \ \text{ and } \ \gamma \in \mP_{\l}( \tau_{1}, \dots, \tau_{k}) \Big\}\;.
\end{equs}
where $ \mP_{\l} $ is the set of possible
pairings among the union of all the leaves of the trees $ \tau_{1}, \dots,
\tau_{k} $. 

We naturally think of elements of $
\mF $ as coloured graphs $ \mG = (\mV, \mE)$, with the pairing $ \gamma $ giving rise to green
edges. Observe
that the forest structure of $ \tau_{1} \sqcup \cdots \sqcup
\tau_{k} $ induces a partial order on the set $ \mV $ of all vertices by
defining $ v > v^{\prime} $ if $ v^{\prime} \neq v $ lies on the unique (directed)
path joining $ v $ to its relative root.
As we mentioned, we will be especially interested in elements of $
\mF $ where each tree $ \tau_{i} $ is an element of $ [\mT_{3}] $, with the
twist that up to two such trees can be planted on the same root:
\begin{equ}[eqn:def-f3]
\text{Each $\tau_i$ consists of either $1$ or $2$ trees in $[\mT_{3}]$
glued at their roots.}
\end{equ}
We then set
\begin{equ}
\mF_{3} = \Big\{ ( \tau_{1}\sqcup \cdots \sqcup \tau_{k}, \gamma) \in \mF 
\text{ such that } \eqref{eqn:def-f3} \text{ holds true }  \}\;.
\end{equ}
Finally, we want to introduce a colouring for $ \mG \in \mF $. We already
mentioned that $ \gamma $ is represented by green
edges among the leaves and that all the roots are coloured red, but we will also
introduce yellow vertices.

\textbf{Colourings.} To simplify the upcoming constructions, let us now write $c(e)\in \{\myarcb, \myarcg\}$ and $c(v)\in
\{\<red>, \<yellow>,\<black>\}$ for the colour of edges and vertices of some graph $
\mG \in \mF $ and
let $K_c$ denote the kernel associated to each edge colour. As we will see
later on, colours will appear only in certain locations and will be
valuated in specific ways. To clarify these points we compile the following
tables for the nodes
\begin{center}
\begin{tabular}{cll}
\toprule
\thead{Colour} & \thead{Where?}  &\thead{Valuation}  \\
\hline
\<red> & Roots & Suprema over spatial variables  \\
\<yellow> & Inner nodes of paths  & Disappear under spatial
integration \\
\<black> & Leaves or untouched inner nodes & Integrate\\
\bottomrule
\end{tabular}\vspace{1em}
\end{center}
and define kernels
\[ K_{\myarcg}  = \eqref{eqn:contraction-kernel} , \qquad
 K_{\myarcb} ((t,x), (s, y))= \begin{cases} P^{1}(t-s, x-y), \quad & \text{ if
} t> s, \\
0 \quad & \text{ else. } \end{cases}  \] 
Now, if $\CG=(\CV,\CE)$, we write $\CV_c$ for the set of vertices of colour
$c$.
Given $z_{\mV} \in (\R^{d+1})^{\CV}$ and $A \subset \CV$, we write $z_A \in (\R^{d+1})^{A}$
for the corresponding projection, as well as $t_A \in \R^{A}$ and $x_A \in (\R^{d})^{A}$
for its temporal and spatial components. In the particular case when $A = \CV_c$ for some
colour $c$ we will sometimes simply write $z_c$ instead of $z_{\CV_c}$. Finally, given an
oriented
edge $e \in \CE$, we write $e = (e_-,e_+)$. 

\textbf{Structure of the estimates.} With these notations at hand, for any subgraph
$ \mG^{\prime} = (\mV^{\prime}, \mE^{\prime} ) \subseteq  \mG \in \mF$, we set
\begin{equ}
K_{\CG^{\prime} }(z) = \prod_{e \in \CE^{\prime}  } K_{c(e)}(z_{e_+} -
z_{e_-})\;,\qquad z \in (\R^{d+1})^{\CV^{\prime} }\;.
\end{equ} 
We can then evaluate the graph at its red roots by
\begin{equs}
\mG (z_{ \<red>}) = \int_{ (\RR^{d+1})^\mI} K_{\mG} (z) \ud z_{\mI} \;,
\end{equs}
where $ \mI = \mV \setminus \mV_{\<red>} $.
 Now, the crux of our bounds is to pick certain spatial variables and use
estimates of the following two kinds, for some $ F, G \geqslant 0 $
\begin{equs}
\int_0^t \int_{\R^d} F(x,s) G(x,s) \ud x \ud s & \leqslant \int_0^t \sup_x F(x,s)
\int_{\R^d} G(x,s) \ud x \ud s \;, \label{eqn:intuitive-bound}\\
 \sup_{x \in \RR^{d}} F(x) G(x) & \leqslant \big(\sup_{x}  F(x) \big) \big(
\sup_{x} G(x) \big) \;. \label{eqn:intuitive-2}
\end{equs}
In our setting $ F, G $ will be two components of $ K_{\mG} $, linked to two
subgraphs $ \mG^{(1)}, \mG^{(2)} \subseteq \mG$, which, appropriately glued
together, form $ \mG $. 

\textbf{Glueing.} In pictures, we describe such a glueing as follows:
\begin{equ}[e:basicglue]
		\begin{tikzpicture}[baseline=8]
			\node[root](root) at (0,0) {};
			\node[dot](l) at (-0.5,0.3) {};
			\node[dot](r) at (.5,.3) {};
			\node[dot](r1) at (.7,.6) {};
			\node[dot](r2) at (.5,.6) {};
			\node[dot](r3) at (.3,.6) {};
			\node[dot](l1) at (-0.7,0.6) {};
			\node[dot](l2) at (-0.5,0.6) {};
			\node[dot](l3) at (-0.3,0.6) {};
			\draw (l) to (root) to (r) {};
			\draw (l) to (l1) {};
			\draw (l) to (l2) {};
			\draw (l) to (l3) {};
			\draw (r) to (r1) {};
			\draw (r) to (r2) {};
			\draw (r) to (r3) {};
			\draw[contract] [bend left=40] (l1) to (r1) {};
			\draw[contract] [bend left=30] (l2) to (r2) {};
			\draw[contract] [bend left=10] (l3) to (r3) {};
		\end{tikzpicture}
=		\begin{tikzpicture}[baseline=5]
			\node[root](root) at (0,0) {};
			\node[innern](l) at (-0.5,0.3) {};
			\node[innern](r) at (.5,.3) {};
			\node[dot](r3) at (.3,.6) {};
			\node[dot](l3) at (-0.3,0.6) {};
			\draw (l) to (root) to (r) {};
			\draw (l) to (l3) {};
			\draw (r) to (r3) {};
			\draw[contract] [bend left=10] (l3) to (r3) {};
		\end{tikzpicture} \ \ \glue \ \ 
		\begin{tikzpicture}[baseline=12]
			\node[root](l) at (-0.5,0.3) {};
			\node[root](r) at (.5,.3) {};
			\node[dot](r2) at (.5,.6) {};
			\node[dot](l2) at (-0.5,0.6) {};
			\draw (l) to (l2) {};
			\draw (r) to (r2) {};
			\draw[contract] [bend left=30] (l2) to (r2) {};
			\node[dot](r1) at (.7,.6) {};
			\node[dot](l1) at (-0.7,0.6) {};
			\draw (l) to (l1) {};
			\draw (r) to (r1) {};
			\draw[contract] [bend left=40] (l1) to (r1) {};
		\end{tikzpicture}
\;.
\end{equ}
When glueing red nodes onto yellow nodes we will be in a setting in which we
use \eqref{eqn:intuitive-bound}. As the notation suggests, we also glue red
nodes together, and in this case we will use the estimate
\eqref{eqn:intuitive-2}. Glueing red nodes is
necessary, say in our running example, to decompose
\begin{equs}
		\begin{tikzpicture}[baseline=12]
			\node[root](l) at (-0.5,0.3) {};
			\node[root](r) at (.5,.3) {};
			\node[dot](r2) at (.5,.6) {};
			\node[dot](l2) at (-0.5,0.6) {};
			\draw (l) to (l2) {};
			\draw (r) to (r2) {};
			\draw[contract] [bend left=30] (l2) to (r2) {};
			\node[dot](r1) at (.7,.6) {};
			\node[dot](l1) at (-0.7,0.6) {};
			\draw (l) to (l1) {};
			\draw (r) to (r1) {};
			\draw[contract] [bend left=40] (l1) to (r1) {};
		\end{tikzpicture}
=
		\begin{tikzpicture}[baseline=12]
			\node[root](l) at (-0.5,0.3) {};
			\node[root](r) at (.5,.3) {};
			\node[dot](r2) at (.5,.6) {};
			\node[dot](l2) at (-0.5,0.6) {};
			\draw (l) to (l2) {};
			\draw (r) to (r2) {};
			\draw[contract] [bend left=30] (l2) to (r2) {};
		\end{tikzpicture}
\glue 		\begin{tikzpicture}[baseline=12]
			\node[root](l) at (-0.5,0.3) {};
			\node[root](r) at (.5,.3) {};
			\node[dot](r1) at (.7,.6) {};
			\node[dot](l1) at (-0.7,0.6) {};
			\draw (l) to (l1) {};
			\draw (r) to (r1) {};
			\draw[contract] [bend left=40] (l1) to (r1) {};
		\end{tikzpicture} \;.
\end{equs}
The glueing procedure we just informally suggested is formalised like this.

\begin{definition}\label{def:glueging-def}
Let $ \mG^{(1)}, \mG^{(2)} \in \mF $ and consider $ A^{(1)} \subseteq \mV_{\<yellow>}^{(1)} \cup
\mV_{\<red>}^{(1)}$, $A^{(2)} \subseteq
\mV_{\<red>}^{(2)}$. 
Given a bijection $\carrow \colon A^{(1)} \to A^{(2)}$, we define $$\CG =
(\mV, \mE)=
\mG^{(1)} \glue \mG^{(2)} = (\CG^{(1)} \sqcup
\CG^{(2)}) / {\carrow}\;,$$ the latter being the equivalence relation
generated by $v \sim \carrow(v)$ for all $v \in A^{(1)}$. 
In this way we also obtain two emebddings $ \mf{e}^{(1)}, \mf{e}^{(2)} $ with
$ \mf{e}^{(i)} \colon \mG^{(i)} \hookrightarrow \mG^{(1)} \glue \mG^{(2)} $,
which we can combine to a map (which is no longer
injective)
$$ \mf{e} \colon \mG^{(1)} \sqcup
\mG^{(2)} \to \mG^{(1)} \glue \mG^{(2)}\;.$$ Finally, the colour assigned to an
identified node $v \in \mV$, i.e.\ some $ v $ belonging to  $ \mf{e} A^{(1)} $, is
\begin{equ}[e:colouring]
c(v)  = \<black>, \ \  \text{ if } v \in \mf{e} \mV^{(1)}_{\<yellow>} \;, \qquad 
c(v)  = \<red>, \ \  \text{ if } v \in \mf{e}\mV^{(1)}_{\<red>} \;.
\end{equ}
\end{definition}
Here and in what follows we write $ \mf{e} A $ short for $ \mf{e}(A) $.
In general, via this procedure the new graph $ \mG $ may not belong to the space of
forests $ \mF $. We therefore 
introduce the following notion of
compatibility between the colours and the structure of $\mF $.
\begin{definition}\label{def:compatible}
A graph $ \mG \in \mF $ has \emph{compatible} colouring if for any $ v \in
\mV $
\begin{equs}
c(v) & = \<red>   \quad & & \Leftrightarrow v \text{ is a root, } \\
c(v) & = \<yellow>  \quad & & \Rightarrow v \text{ is neither a leaf nor a root,
i.e.\ an inner node.}
\end{equs}
\end{definition}

\begin{remark}
Note that the rooted forest structure of $\CG$ is uniquely determined by its colouring of
edges (black or green) and vertices (black, red, or yellow).
\end{remark}

From here on we always work with \emph{compatible} graphs. In particular, we
have the following result.
\begin{lemma}\label{lem:glueging-def}
Given compatible graphs $ \mG^{(1)}, \mG^{(2)} \in \mF $ and $\carrow$
as in Definition~\ref{def:glueging-def}, $\mG^{(1)} \glue \mG^{(2)}$ 
is a compatible element of $ \mF $, with the same
pairing as $ \mG^{(1)} \sqcup \mG^{(2)} $.
\end{lemma}

\begin{proof}
Since $ \mV_{\<red>}^{(2)} $ consists of roots and since these are
always glued onto either roots or inner nodes by Definition~\ref{def:compatible},
$\mG^{(1)} \glue \mG^{(2)}$ again has a natural forest structure. The rule \eqref{e:colouring} furthermore
guarantees that the compatibility property
is preserved. (Note also that no new yellow nodes are created.)
\end{proof}


\textbf{Valuation.} Now we define a valuation $ | \mG | $ of a compatible graph $ \mG \in \mF $ as follows:
\begin{claim}
\item Consider the kernel $ K_{\mG} $ defined by the edges of $\CG$ and integrate it over all 
the black variables (i.e.\ variables associated to black vertices), 
with all remaining variables fixed.
\item Second, integrate over all the spatial components of the yellow variables.
\item Third, take the supremum over the temporal components of the yellow variables.
\item Finally, take the supremum over the spatial components of red variables.
\end{claim}
Clearly, the result of such a valuation will depend only on the red temporal variables $ | \mG | = |
\mG |(t_{\<red>}) $. In formulae
\begin{equ}
|\CG|(t_{\<red>}) = \sup_{x_{\<red>},
t_{\<yellow>}} \int K_\CG(z) \ud z_{\<black>} \ud
x_{\<yellow>}\;.
\end{equ}
\begin{remark}\label{rem:evaluation}
This definition is consistent with our evaluation of the graph at a given
point in the sense that $ \mG (z_{\<red>}) \leqslant  |\mG|(t_{\<red>}) $ if all
nodes of $ \mG $ apart from the red roots are coloured black.
\end{remark}
Now, our aim will be to combine the valuation with the glueing of two graphs
and obtain an estimate of the sort $ | \mG^{(1)} \glue \mG^{(2)} | \leqslant |
\mG^{(1)} | \cdot | \mG^{(2)} | $. Of course, in this form the estimate cannot be true, 
since the
right-hand side has more free (red) variables than the left-hand side; these additional
free variables will be integrated over. In addition, in order to obtain a
useful bound we have to also take care of the time supremum over $
t_{\<yellow>} $: here it will be natural to assume that all yellow nodes are
covered by the glueing procedure, and we call such a glueing \emph{onto}. 
we capture this in the following notation.

\begin{definition}\label{def:onto}
Consider two compatible graphs $ \mG^{(1)} , \mG^{(2)} \in \mF $ and $ A^{(1)} \subseteq
\mV_{\<yellow>}^{(1)} \cup \mV_{\<red>}^{(1)}, A^{(2)} \subseteq
\mV^{(2)}_{\<red>} $, together with a bijection $ \carrow \colon A^{(2)} \to
A^{(1)} $. 
If $$ \mV_{\<yellow>}^{(1)} \subseteq 
A^{(1)}, \qquad \mV_{\<yellow>}^{(2)}= \emptyset \;, $$
then we say that $ \carrow $ defines an \textbf{\emph{onto}} glueing of $ \mG^{(2)} $ onto $
\mG^{(1)} $. 
In particular, for $ \mG = \mG^{(1)} \glue \mG^{(2)} = (\mV, \mE) $ it holds
that $ \mV_{\<yellow>} = \emptyset $.
\end{definition}
Now the key result for our bounds is the following lemma. Here we introduce the
temporal integration bound
$\mf{t} = \max_{v \in \mV_{\<red>}} t_{v}$.

\begin{lemma}\label{lem:product}
Consider two compatible $ \mG^{(1)}, \mG^{(2)} \in \mF $ and an onto glueing $
\carrow $. Then for $ \mG = \mG^{(1)} \glue \mG^{(2)} $ we have
\begin{equs}
| \mG |(t_{\mV_{\<red>}})
 \leqslant | \mG^{(1)} | ( t_{ \mf{e} \mV_{\<red>}^{(1)}}) \cdot
\int_{[0, \mf{t}]^{\mf{e}\mV^{(1)}_{\<yellow>}}} |
\mG^{(2)} | (t_{ \mf{e} \mV^{(2)}_{\<red>}})\ud t_{\mf{e} \mV^{(1)}_{\<yellow>}}
\;.
\end{equs}
\end{lemma}

\begin{remark}\label{rem:notation} We observe that the temporal variables 
on which the valuation depends are inherited from the glued graph $ \mG
$, thus determining a pairing among variables of $ \mG^{(1)} $ and $
\mG^{(2)} $. This is captured by the map $ \mf{e} $. We note in
particular that $ \mf{e} \mV_{\<yellow>}^{(1)} \subseteq \mf{e}
\mV^{(2)}_{\<red>} $ since the glueing is onto, as well as $ \mf{e}
\mV_{\<red>}^{(1)} \subseteq \mV_{\<red>} $. In fact, we can write $
\mf{e} \mV_{\<red>}^{(2)} $ as the disjoint union
\[ \mf{e} \mV_{\<red>}^{(2)} = \left( \mf{e} \mV_{\<yellow>}^{(1)} \cap
\mf{e}\mV_{\<red>}^{(2)} \right)
\sqcup \left(  \mV_{\<red>} \cap \mf{e} \mV_{\<red>}^{(2)}  \right) \;. \] 
\end{remark}

\begin{remark}\label{rem:multiplicative}
A special case is given by $\mV_{\<yellow>}^{(1)} = \mV_{\<yellow>}^{(2)} = \emptyset$
in which case one has $\CV_{\<red>} = \CV_{\<red>}^{(1)} \sqcup \CV_{\<red>}^{(1)}$
and $|\CG| \le |\CG^{(1)}|\cdot |\CG^{(2)}|$.
\end{remark}

\begin{proof}
Since by assumption $ \carrow $ is onto, we have $ \mV_{\<yellow>} =
\emptyset $ and thus
\begin{equs}
| \mG | ( t_{\mV_{\<red>}}) = \sup_{x_{\CV_{\<red>}}} \int K_\CG(z) \ud z_{\CV_{\<black>}} \;.
\end{equs}
Note first that we can write
\begin{equ}
\mV_{\<black>} = \mf{e} \mV_{\<black>}^{(1)} \sqcup \mf{e}\mV_{\<black>}^{(2)} \sqcup
\mf{e}\mV_{\<yellow>}^{(1)}\;,
\end{equ}
by the assumptions on
the colouring of $ \mG^{(1)}, \mG^{(2)} $ (in particular, recall that $
\mf{e} \mV_{\<yellow>}^{(1)} \subseteq \mV_{\<black>} $ because we colour nodes black
after glueing).
As a consequence, we have
\begin{equs}
\int K_{\mG}(z) \ud z_{\mV_{\<black>}} = \int \bigg( \int K_{\mG^{(1)}}
(z) \ud z_{\mf{e}\mV_{\<black>}^{(1)}}  \bigg) \cdot \bigg( \int
K_{\mG^{(2)}}(z) \ud z_{\mf{e} \mV_{\<black>}^{(2)}} \bigg) \ud
z_{\mf{e} \mV_{\<yellow>}^{(1)}} \;,
\end{equs}
where we used that $ \mf{e} \mV_{\<black>}^{(1)} $ and $
\mf{e} \mV_{\<black>}^{(2)} $ do not intersect. It then follows from
\eqref{eqn:intuitive-bound} that
\begin{equs}
\int & K_{\mG}(z) \ud z_{\mV_{\<black>}} \\
& = 
\int_{\RR^{\mf{e}\mV_{\<yellow>}^{(1)}}} 
 \int_{(\RR^{d})^{\mf{e} \mV^{(1)}_{\<yellow>}}}\bigg( \int K_{\mG^{(1)}}
(z) \ud z_{\mf{e}\mV_{\<black>}^{(1)}}  \bigg) \cdot \bigg( \int
K_{\mG^{(2)}}(z) \ud z_{\mf{e}\mV_{\<black>}^{(2)}} \bigg) \ud
x_{\mf{e}\mV_{\<yellow>}^{(1)} }\ud t_{\mf{e}\mV_{\<yellow>}^{(1)}} \\
& \leqslant  \int_{\RR^{\mf{e} \mV_{\<yellow>}^{(1)}}} 
 \bigg(\int K_{\mG^{(1)}}
(z) \ud z_{\mf{e}\mV_{\<black>}^{(1)}}   \ud
x_{\mf{e}\mV_{\<yellow>}^{(1)} }\bigg) \cdot \bigg( \sup_{x_{\mf{e}\mV_{\<yellow>}^{(1)}}} \int
K_{\mG^{(2)}}(z) \ud z_{\mf{e}\mV_{\<black>}^{(2)}} \bigg)\ud t_{\mf{e}\mV_{\<yellow>}^{(1)}} 
\;.
\end{equs}
Next, observe that $ \mf{e}\mV^{(1)}_{\<yellow>} \subseteq
\mf{e}\mV^{(2)}_{\<red>} $, and that from the definition of the kernel the time
variables are ordered so that $ t_{v} \in [0, \mf{t}] $ for any $
v \in \mV $. In particular the previous estimate is bounded by
\begin{equs}
\int & K_{\mG}(z) \ud z_{\mV_{\<black>}} \\
& \leqslant \int_{[0, \mf{t}]^{\mf{e} \mV_{\<yellow>}^{(1)}}} 
\bigg(\int K_{\mG^{(1)}}
(z) \ud z_{\mf{e}\mV_{\<black>}^{(1)}}   \ud
x_{\mf{e}\mV_{\<yellow>}^{(1)} }\bigg) \cdot \bigg( \sup_{x_{\mf{e}\mV_{\<red>}^{(2)}}} \int
K_{\mG^{(2)}}(z) \ud z_{\mf{e}\mV_{\<black>}^{(2)}} \bigg)\ud t_{\mf{e}\mV_{\<yellow>}^{(1)}} 
\\
&\leqslant 
| \mG^{(1)}|(t_{\mf{e}\mV^{(1)}_{\<red>}})
\int_{[0, \mf{t}]^{\mf{e} \mV_{\<yellow>}^{(1)}}} 
| \mG^{(2)}| ( t_{\mf{e}\mV^{(2)}_{\<red>}}) \ud t_{\mf{e}\mV_{\<yellow>}^{(1)}}\;,
\end{equs}
as required.
\end{proof}


\textbf{Running example.} Finally, we can use Lemma~\ref{lem:product} to
conclude the estimates for \eqref{eqn:32nd}. It is important to remark that in all cases of our interest the supremum over
$ t_{\mV_{\<yellow>}^{(1)}} $ in the valuation of $ \mG^{(1)} $ will be superfluous, by an application of the semigroup
property. This can be swiftly seen by proceeding in the calculation for our
running example. Since no confusion can occur, we now write 
$t$ instead of $\mf{t}$, observing that the latter
does indeed coincide in our case of interest with the temporal variable of the 
only red node present in the original graph. By \eqref{e:basicglue} and Lemma~\ref{lem:product},
we obtain
\begin{equs}
		\begin{tikzpicture}[baseline=8]
			\node[root](root) at (0,0) {};
			\node[dot](l) at (-0.5,0.3) {};
			\node[dot](r) at (.5,.3) {};
			\node[dot](r1) at (.7,.6) {};
			\node[dot](r2) at (.5,.6) {};
			\node[dot](r3) at (.3,.6) {};
			\node[dot](l1) at (-0.7,0.6) {};
			\node[dot](l2) at (-0.5,0.6) {};
			\node[dot](l3) at (-0.3,0.6) {};
			\draw (l) to (root) to (r) {};
			\draw (l) to (l1) {};
			\draw (l) to (l2) {};
			\draw (l) to (l3) {};
			\draw (r) to (r1) {};
			\draw (r) to (r2) {};
			\draw (r) to (r3) {};
			\draw[contract] [bend left=40] (l1) to (r1) {};
			\draw[contract] [bend left=30] (l2) to (r2) {};
			\draw[contract] [bend left=10] (l3) to (r3) {};
		\end{tikzpicture}
\leqslant \left\vert
		\begin{tikzpicture}[baseline=5]
			\node[root](root) at (0,0) {};
			\node[innern](l) at (-0.5,0.3) {};
			\node[innern](r) at (.5,.3) {};
			\node[dot](r3) at (.3,.6) {};
			\node[dot](l3) at (-0.3,0.6) {};
			\draw (l) to (root) to (r) {};
			\draw (l) to (l3) {};
			\draw (r) to (r3) {};
			\draw[contract] [bend left=10] (l3) to (r3) {};
		\end{tikzpicture} \right\vert (t)  \cdot \int_{[0, t]^2} \left\vert
		\begin{tikzpicture}[baseline=12]
			\node[root](l) at (-0.5,0.3) {};
			\node[root](r) at (.5,.3) {};
			\node[dot](r2) at (.5,.6) {};
			\node[dot](l2) at (-0.5,0.6) {};
			\draw (l) to (l2) {};
			\draw (r) to (r2) {};
			\draw[contract] [bend left=30] (l2) to (r2) {};
			\node[dot](r1) at (.7,.6) {};
			\node[dot](l1) at (-0.7,0.6) {};
			\draw (l) to (l1) {};
			\draw (r) to (r1) {};
			\draw[contract] [bend left=40] (l1) to (r1) {};
	\end{tikzpicture} \right\vert (
t_{\<red>}) \ud t_{\<red>} \;.
\end{equs}
Regarding the first factor, it follows from the semigroup property and the fact that
we integrate over all yellow spatial variables that
\begin{equ}[eqn:def-s]
\left|		\begin{tikzpicture}[baseline=5]
			\node[root](root) at (0,0) {};
			\node[innern](l) at (-0.5,0.3) {};
			\node[innern](r) at (.5,.3) {};
			\node[dot](r3) at (.3,.6) {};
			\node[dot](l3) at (-0.3,0.6) {};
			\draw (l) to (root) to (r) {};
			\draw (l) to (l3) {};
			\draw (r) to (r3) {};
			\draw[contract] [bend left=10] (l3) to (r3) {};
		\end{tikzpicture}
\right| (t)
= \left|		\begin{tikzpicture}[baseline=5]
			\node[root](root) at (0,0) {};
			\node[dot](r3) at (.3,.6) {};
			\node[dot](l3) at (-0.3,0.6) {};
			\draw (root) to (l3) {};
			\draw (root) to (r3) {};
			\draw[contract] [bend left=10] (l3) to (r3) {};
		\end{tikzpicture}
\right| (t) \leqslant  c^{2} (e^{t} \varepsilon^{\frac d2-
\alpha})^2( t\vee \varepsilon^2)^{-\frac d2} \eqdef
\mf{s}_{\ve}(t)\;,
\end{equ}
for come $ c > 0 $. Hence, using Remark~\ref{rem:multiplicative}, we have
\begin{equs}
\bigg|
		\begin{tikzpicture}[baseline=8]
			\node[root](root) at (0,0) {};
			\node[dot](l) at (-0.5,0.3) {};
			\node[dot](r) at (.5,.3) {};
			\node[dot](r1) at (.7,.6) {};
			\node[dot](r2) at (.5,.6) {};
			\node[dot](r3) at (.3,.6) {};
			\node[dot](l1) at (-0.7,0.6) {};
			\node[dot](l2) at (-0.5,0.6) {};
			\node[dot](l3) at (-0.3,0.6) {};
			\draw (l) to (root) to (r) {};
			\draw (l) to (l1) {};
			\draw (l) to (l2) {};
			\draw (l) to (l3) {};
			\draw (r) to (r1) {};
			\draw (r) to (r2) {};
			\draw (r) to (r3) {};
			\draw[contract] [bend left=40] (l1) to (r1) {};
			\draw[contract] [bend left=30] (l2) to (r2) {};
			\draw[contract] [bend left=10] (l3) to (r3) {};
		\end{tikzpicture}
\bigg| 
& \leqslant \mf{s}_{\ve}(t) \int_{[0, t]^2} 
 \left\vert
		\begin{tikzpicture}[baseline=12]
			\node[root](l) at (-0.5,0.3) {};
			\node[root](r) at (.5,.3) {};
			\node[dot](r2) at (.5,.6) {};
			\node[dot](l2) at (-0.5,0.6) {};
			\draw (l) to (l2) {};
			\draw (r) to (r2) {};
			\draw[contract] [bend left=30] (l2) to (r2) {};
		\end{tikzpicture}  \right\vert (t_{\<red>})
\cdot \left\vert \begin{tikzpicture}[baseline=12]
			\node[root](l) at (-0.5,0.3) {};
			\node[root](r) at (.5,.3) {};
			\node[dot](r1) at (.7,.6) {};
			\node[dot](l1) at (-0.7,0.6) {};
			\draw (l) to (l1) {};
			\draw (r) to (r1) {};
			\draw[contract] [bend left=40] (l1) to (r1) {};
		\end{tikzpicture} \right\vert \,
 (t_{\<red>}) \, \ud t_{\<red>}
\\
 &  \leqslant \mf{s}_{\ve}(t) \int_{[0, t]^2}
\mf{s}_{\ve} (t_{\<red>,1})\mf{s}_{\ve} (t_{\<red>,2})
\ud t_{\<red>}\\
&  \leqslant (t \vee \ve^{2})^{- \frac{d}{2}} (c e^{t}\varepsilon^{\frac d2- \alpha})^6
\left(\int_0^{t}(s\vee
\varepsilon^2)^{-\frac d2} \ud s\right)^2 \eqdef B_{\ve}(t) \;.
\end{equs}
The second graph on the right-hand side of \eqref{eqn:32nd} can be bounded by
$B_\eps(t)$ in an analogous way.
We recall that every smaller tree is evaluated at the temporal variable it has
inherited from the original large tree (so the last two trees have the same
contribution, but they are evaluated at different time variables). We omit writing the dependence on such
variables for convenience.
We conclude that
	\begin{equation}\label{est:EIY3}
		\E\<IY^3>^2(t,x) 
		\le15 B_\eps(t) \;.
	\end{equation}
Here, the number $15=6+9$
is the total number of ways of pairing $6$ green nodes.
In addition, we observe that $ B_{\ve}(t) $ can be rewritten, by a change of
variables, as
\begin{equ}
B_{\ve}(t) =(t \vee \ve^{2})^{- \frac{d}{2}} (c e^{t}\varepsilon^{\frac d2-
\alpha})^2 \bigg( c e^{t} \ve^{1- \alpha } \bigg( \int_0^{t \ve^{-2}}(s\vee
1)^{-\frac d2} \ud s \bigg)^{\frac{1}{2}} \bigg)^4  \;,
\end{equ}
where $ 4 = 2 (\ell(\tau) -1) $, for the trees $ \tau $ that we took under
consideration.
We now claim that an analogous bound holds for every tree $\tau$. More precisely, we have the following result.
		\begin{proposition}\label{prop:momentX} For every $\tau$ in $\tree_3$
and $ q \in [2, \infty) $,
		\begin{equ}\label{est:mXtau}
			\|X^\tau_{\ve}(t,x)\|_{L^q(\Omega)}\lesssim  (t\vee \varepsilon^2)^{-\frac d4}e^t \varepsilon^{\frac d2- \alpha} \(e^t \varepsilon^{1- \alpha}\Gamma_{d/2}^{1/2}(t\varepsilon^{-2})\)^{\ell(\tau)-1}
		\end{equ}
		and
		\begin{equ}\label{est:mDXtau}
			\|\nabla X^\tau_{\ve}(t,x)\|_{L^q(\Omega)}\lesssim  (t\vee \varepsilon^2)^{-\frac d4-\frac12}e^t \varepsilon^{\frac d2- \alpha} \(e^t \varepsilon^{1- \alpha}\Gamma_{d/2}^{1/2}(t\varepsilon^{-2})\)^{\ell(\tau)-1}\,,
		\end{equ}
		where for each $a>0$, $\Gamma_a(t)=\int_0^t(s\vee 1)^{-a} \ud s$.
	\end{proposition}
\begin{remark}
We observe that for $ t \geqslant \ve^{2}$ we have 
\begin{equ} 1 \leqslant \Gamma_{\frac{d}{2}} (t \ve^{-2}) \leqslant 1 +
\begin{cases} \log{(t \ve^{- 2})} \quad & \text{ if } d=2 , \\
 (1 - \frac{2}{d})^{-1} \quad & \text{ if } d > 2, \end{cases}
\end{equ}
which leads to the time scale $ t_{1}(\ve)$ (in particular to the constraint $\bar \alpha < 1$), 
since the estimates in
Proposition~\ref{prop:momentX} improve as $ \l(\tau) $ increases only for times
$ t \leqslant t_{1}(\ve) $.
\end{remark}

	\begin{proof}
		By hypercontractivity, it suffices to show \eqref{est:mXtau} for $q=2$ since each $X^\tau_{\ve}$ belongs to a finite Wiener chaos. 
		Let $\tau$ be a tree in $\tree_3$ and $\bar \tau$ be an identical copy of $\tau$. Let $\gamma$ be a way of pairing the leaves of $[\tau,\bar\tau]$ such that each leaf is paired with exactly one other leaf. The coloured graph $[\tau,\bar\tau]_\gamma$ is obtained by connecting the leaves of $[\tau,\bar\tau]$ with green edges according to the pairing rule $\gamma$ and connecting the two roots to a red node, so that
		\begin{equation*}
			\E |X^\tau_{\ve}(t,x)|^2=\sum_{\gamma}\
[\tau,\bar\tau]_\gamma(z_{\<red>}) \le \sum_{\gamma}\
|[\tau,\bar\tau]_\gamma|(t_{\<red>})\,.
		\end{equation*}
 As usual, the value associated to the single red node is fixed to $(t,x)$ and since
no confusion can occur we write $t$ instead of $t_{\<red>}$.
		Hence, the estimate \eqref{est:mXtau} for $q=2$ is obtained once we show that for any pairing rule $\gamma$, 
		\begin{equation}\label{est:ttg}
			|[\tau,\bar\tau]_\gamma| (t)\le (t\vee
\varepsilon^2)^{-\frac d2}(ce^{t} \varepsilon^{\frac d2-
\alpha})^{2\ell(\tau)}\(\int_0^{t} (s\vee \varepsilon^2)^{-\frac d2}
\ud s \)^{\l(\tau)-1}\,,
		\end{equation}
where $ | \cdot | $ is the valuation introduced above. Let us write $ \mG = (
\mV, \mE) \in \mF $ for the compatible tree $ [ \tau, \overline{\tau}]_{\gamma} $.

To prove \eqref{est:ttg} we introduce an algorithm that allows us to iterate
the type of bounds we described previously. We start 
by constructing a succession
$ \{ p_{i} \}_{i=0}^{N-1} $ of self-avoiding paths $ p_{i} = (\mV_{p_{i}},
\mE_{p_{i}}) \subseteq (\mV, \mE)$  that cover the entire tree $
\mG $.

By self-avoiding path we mean a path with
endpoints $ p_{-}, p_{+} \in \mV $ (with possibly $ p_{-} = p_{+} $) such that all points,
apart from possibly $ p_{-},p_{+} $ appear exactly once in an edge (in
particular, all edges are distinct). In addition we ask
that the endpoints are minimal points with respect to the ordering inherited by
the forest structure of $ \mG $. 

The paths $ \{ p_{i} \}_{i = 0}^{N-1} $ will cover the entire tree $ \mG $ in
the sense that we will iteratively define compatible graphs $ \mG^{(i)} \in
\mF_{3} $ for
$ i \in \{ 0, \dots, N \} $ by onto glueings, satisfying
\[ \mG^{(N)} = \emptyset\;, \quad \mG^{(i-1)} = p_{i-1} \glue \mG^{(i)}
\;, \quad \mG^{(0)} = \mG\;.  \]
In particular it is useful to note that we obtain
the following chain of embeddings
\begin{equs}
\begin{tikzcd}[cells={nodes={minimum height=2em}}]
 \mG^{(N)} \arrow[r, "\mf{e}_{N-1}"] & \mG^{(N-1)}\arrow[r,"\mf{e}_{N-2}"]  &
\cdots \cdots  \arrow[r,"\mf{e}_{1}"]  &
\mG^{(1)} \arrow[r,"\mf{e}_{0}"] & \mG^{(0)} \\
 & p_{N-1} \arrow[u,"\mf{e}_{N-1}"]  & p_{i} \arrow[u,"\mf{e}_{i}"]& p_{1}\arrow[u,"\mf{e}_{1}"] &  p_{0}\arrow[u,"\mf{e}_{0}"] 
\end{tikzcd} \;,
\end{equs}
where each $ \mf{e}_{i} $ is the map of
Definition~\ref{def:glueging-def}. Thus we obtain an overall map
\begin{equ}[eqn:def-e]
\mf{e} \colon p_{0} \sqcup \cdots \sqcup p_{N-1} \to \mG \;,
\end{equ}
defined by composing $ \mf{e} = \mf{e}_{0} \circ \cdots \circ \mf{e}_{i} $ on $
p_{i} $. The way the paths $ p_{i} $ are chosen is described by the following procedure, starting from the graph $
\mG^{(0)}= \mG $ introduced above.

\begin{enumerate}


\item \textbf{Construction of the path.} Assume that for some $ i \in \NN \cup
\{ 0 \} $ we are given a compatible graph $
\mG^{(i)} = (\mV^{(i)}, \mE^{(i)}) \in \mF_{3}$, with $ c(v) \in \{ \<black>,
\<red> \} $ for all $ v \in \mV^{(i)} $ (note that this is the case for $
\mG^{(0)} $, since we have one red root and no node is coloured yellow). We define $ p_{i} $ by choosing as a starting
point any $ v \in \mV^{(i)} $ such that $ c(v) =\<red> $, namely any
root of $ \mG^{(i)}$. 

By assumption, since $ \mG^{(i)} $ is an element of $ \mF_{3} $, we can
follow a path up through any one of the two maximal trees rooted at $ v
$ and let us call the sub-tree we choose $ \tau(v) $. Then, because of oddity, we can choose to arrive at a leaf such that the
pairing $ \gamma $ leads to another leaf, belonging to some $
\overline{\tau}(v) \neq \tau(v) $. We then
continue the path by descending the maximal subtree this leaf belongs to, until
we reach its red root
(this may be again $ v $, since $ v $ can be the root of up to two trees). We
colour all the nodes of $ p_{i} $ yellow, apart from its endpoints and leaves,
which are left red and black respectively.


\item \textbf{Removal.} We then define the compatible graph $ \mG^{(i+1)} = (\mV^{(i+1)},
\mE^{(i)} \setminus \mE_{p_{i}}) \in \mF$ with the colourings inherited by $
\mG^{(i)} $. Here $ \mV^{(i+1)} $ is obtained
from $ \mV^{(i)} $ by removing all singletons (meaning the two leafs connected by
the only green edge $ p_{i} $ has crossed and possibly the endpoints
of $ p_{i} $, if any of them was the root of just one tree, or $ p_{i} $ was a
cycle). We finally colour $ c(v) = \<red> $ for $
v \in \mV^{(i+1)} \cap \mV_{p_{i}} $. In particular \eqref{eqn:def-f3} still
holds true for $ \mG^{(i+1)} $, so that $ \mG^{(i+1)} $ lies in $ \mF_{3} $. Moreover,
we obtain that $ \mG^{(i)} $ is the \emph{onto} glueing of $ \mG^{(i+1)} $ onto $ p_{i} $
\begin{equs}
\mG^{(i)} = p_{i} \glue \mG^{(i+1)} \;,
\end{equs}
with the map $ \carrow $ defined by inverting the removal we just described.

\item \textbf{Conclusion.} We can now repeat steps $ 1 $ and $ 2 $ for $ N
$ times until $ \mG^{(N)} = \emptyset $ (if not, there must still be a red
root available and we can still follow the previous algorithm). Since at every step
we are removing exactly one green edge, we see that this algorithm
terminates in $ N = \l(\tau) $ steps.
\end{enumerate}
We are left with a collection of $ N $ paths $ \{ p_{i} \}_{i =
0}^{\l(\tau) -1} $ such
that
\begin{enumerate}
\item $ \mE $ is the disjoint union $ \mE = \bigsqcup_{i =1}^{\l(\tau)-1}
\mE_{p_{i}} $.
\item Every $ v \in \mV $ which is not a leaf appears in exactly thee paths,
twice as a root (coloured red) and once as an inner node (coloured yellow).
Instead, every leaf appears in exactly one path, coloured black.
\end{enumerate}
We observe that by the semigroup property, every path $ p_{i} $ yields a contribution
\begin{equ}[eqn:bd-path]
 | p_{i} | (t_{\<red>}) \leqslant \begin{cases} \sqrt{\mf{s}_{\ve}
(t_{p_{i,+}})\mf{s}_{\ve}
(t_{p_{i,-}})} & \text{ if } p_{i,+} \neq p_{i,-}, \\
 \mf{s}_{\ve} (t_{p_{i}}) & \text{ if } p_{i,+} = p_{i,-} = p_i,
\end{cases}
\end{equ}
where $ \mf{s}_{\ve} $ is defined in \eqref{eqn:def-s}.
To conclude, we may now use Lemma~\ref{lem:product} to obtain
\begin{equs}
 \mG (t, x) &  \leqslant  \mf{s}_{\ve}(t)\int_{[0, t]^{\mf{e} \mV_{
p_{0}, \<yellow>}} } |
\mG^{(1)} | (t_{\mf{e}\mV_{\<red>}^{(1)}}) \ud t_{\mf{e} \mV_{p_{0},
\<yellow>}}\;,
\end{equs}
where the map $ \mf{e} $ is defined in \eqref{eqn:def-e}. And similarly,
for $ i \in \{ 1, \dots, N-1 \} $, using the definition of the map $
\mf{e} $, we find
\begin{equs}
| \mG^{(i)} |(t_{\mf{e} \mV^{(i)}_{\<red>}}) \leqslant
| p_{i} |(t_{\mf{e} \mV_{p_{i}, \<red>}}) \int_{[0, t]^{\mf{e} \mV_{
p_{i}, \<yellow>}} } |
\mG^{(i+1)} | (t_{\mf{e}\mV_{\<red>}^{(i+1)}}) \ud t_{\mf{e} \mV_{p_{i}, \<yellow>}}\;,
\end{equs}
so that overall
\begin{equs}
 \mG (t, x) &  \leqslant \mf{s}_{\ve}(t)\int_{[0, t]^{ \overline{\mI}}}
\prod_{i =1}^{\l(\tau) -1} | p_{i} |(t_{\mf{e}\mV_{p_{i} , \<red>}})
\ud t_{\overline{\mI}}  \;,
\end{equs}
where $ \overline{\mI} = \sqcup_{i = 0}^{N-1} \mf{e} \mV_{p_{i}, \<yellow>} \subseteq \mV $ is the set of inner nodes that are not
leaves, in view of the considerations above. Since in addition by our observations every node in $ \overline{\mI} $ appears
exactly twice as a root of some path (either a cycle or two different paths),
by \eqref{eqn:bd-path} we finally obtain
\begin{equs}
\mG (t,x) &  \leqslant \mf{s}_{\ve}(t)\int_{[0, t]^{ \overline{\mI}}} \mf{s}_{\ve}(t_{ \overline{\mI}}) \ud t_{
\overline{\mI}} \leqslant \mf{s}_{\ve}(t) \bigg(
\int_{0}^{t} \mf{s}_{\ve} (s) \ud s \bigg)^{| \overline{\mI} |}\\
& \leqslant (t \vee \ve^{2})^{- \frac{d}{2}} (c e^{t} \ve^{\frac{d}{2} - \alpha})^{2 \l(\tau)} \bigg( \int_{0}^{t}
(s \vee \ve^{2})^{- \frac{d}{2}} \ud s \bigg)^{\l(\tau)-1}
 \;,
\end{equs}
which is \eqref{est:ttg}. Here we have used that $ | \overline{\mI} | = 2 i
(\tau) $ since we have two copies of the same tree, and $ 2i(\tau) =
\ell(\tau) -1 $.

As for \eqref{est:mDXtau}, we can follow the
same argument as above, to write:
\begin{equs}
 \EE | \nabla X^{\tau}_{\ve} (t, x) | = \sum_{\tau}\ [ \nabla \tau, \nabla
\overline{\tau}]_{\gamma} (t,x),
\end{equs}
where $ \nabla \tau $ indicates the same convolution integral associated to the tree
$\tau$ as above, only with the kernel $ - e^{t} \nabla p_{t}(x) $ used in the
edge connecting to the root. Graphically, if $ \tau $ is built starting from the
trees $ \tau_{1}, \tau_{2}, \tau_{3} $ we have
	\begin{equation*}
		\tau= 
		\begin{tikzpicture}[baseline=2]
			\node[dot] (0) at (0,0) {};
			\draw (0) to (0, 0.3);
			\draw (0,0.3) to (-.5,0.6) node at (-.5,.8) {$\tau_1$} ;
			\draw (0,0.3) to (0,0.6) node at (0,.8) {$\tau_2$} ;
			\draw (0,0.3) to (.5,0.6) node at (.5,.8) {$\tau_{3}$};
		\end{tikzpicture}\, , \qquad
		\nabla \tau= 
		\begin{tikzpicture}[baseline=2]
			\node[dot] (0) at (0,0) {};
			\draw[densely dotted] (0) to (0, 0.3);
			\draw (0,0.3) to (-.5,0.6) node at (-.5,.8) {$\tau_1$} ;
			\draw (0,0.3) to (0,0.6) node at (0,.8) {$\tau_2$} ;
			\draw (0,0.3) to (.5,0.6) node at (.5,.8) {$\tau_{3}$};
		\end{tikzpicture}\, ,
	\end{equation*}
where the dotted line indicates convolution against the kernel $  - e^{t} \nabla p_{t}(x)$.
In particular, we can decompose the graph $ [ \nabla \tau, \nabla
\overline{\tau}]_{\gamma} $ into the same set of paths $ \{ p_{i}
\}_{i =0}^{N-1} $ as for $
[\tau, \overline{\tau}]_{\gamma} $. The only difference is that now the path $
p_{0} $ contains two dotted lines coming into the root: all other paths
remain unchanged. By the semigroup property the contribution of the path $
p_{0} $ is the same as that of
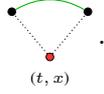
\begin{equ}
	\begin{tikzpicture}[baseline=5]
			\node[root,label=below:{\tiny$(t,x)$}](root) at (0,0) {};
			\node[dot](left) at (-0.5,.6) {};
			\node[dot](right) at (.5,.6) {};
			\draw[densely dotted] (root) to (left) {};
			\draw[densely dotted] (root) to (right) {};
			\draw[contract] (right) [bend right=30] to (left) {};
		\end{tikzpicture} \, .
\end{equ}
Hence, using the bound (for some $ c > 0 $)
\begin{equs}
\| \nabla p^{\ve}_{t} \|_{L^{2}(\RR^{d})} & \leqslant (\| \nabla p_{t}
\|_{L^{2}(\RR^{d})} \| \varphi^{\ve} \|_{L^{1}(\RR^{d})} )
\wedge ( \| p_{t} \|_{L^{1}(\RR^{d})} \| \nabla \varphi^{\ve}
\|_{L^{2}(\RR^{d})} ) \\
& \leqslant c ( t \vee \ve^{2} )^{-\frac{d}{4} -
\frac{1}{2}}\;,
\end{equs}
we find that the contribution of $ p_{0} $ is bounded from above by
\begin{equs}
 (c \varepsilon^{\frac d2- \alpha})^2 e^{ 2t} (t \vee
\varepsilon^2)^{-\frac d2- 1} \;,
\end{equs}
which together with the previous calculations yields \eqref{est:mDXtau}.
\end{proof}

We can extend the previous result to longer timescales as follows.
	\begin{proposition}\label{prop:mPXtau}
		For every $ q \in [2, \infty ) $, $\tau$ in $\tree_3$ and every
$t>t_1(\varepsilon)$, the following estimates hold true
	\begin{equs}
	\|P^1_{t-t_1}X^\tau_{\ve}(t_1,x)\|_{L^q(\Omega)} &\lesssim (t\vee \varepsilon^2)^{-\frac d4}e^t \varepsilon^{\frac d2- \alpha} \(e^{t_1}
\varepsilon^{1- \alpha}\Gamma_{d/2}^{1/2}(t_1\varepsilon^{-2})\)^{\ell(\tau)-1},\label{est:mPXtau} \\		
\|\nabla P^1_{t-t_1}X^\tau_{\ve} (t_1,x)\|_{L^{q}(\Omega)} & \lesssim (t \vee
\ve^{2})^{-\frac d4-\frac12}e^t \varepsilon^{\frac d2-
\alpha}\( e^{t_1}\varepsilon^{1- \alpha}\Gamma_{d/2}^{1/2}(t_{1} \ve^{-2})
\)^{\ell(\tau)-1}. \qquad \quad  \label{est:mPDXtau}
		\end{equs}
	\end{proposition}
	\begin{proof}
By hypercontractivity it suffices to prove the result for $ q =2 $.
We can change the graphical notation we introduced previously by adding blue inner
nodes \tikz[baseline=-2]\node[bluedot] {}; associated to a time $ t_{1}>0 $ to
indicate integration against the measure $ - \delta_{t_{1}}(s) \ud s \ud y $:
once more, the minus sign is a matter of convention. It is natural though since our lines
represent integration against $ - P^{1} $ as a consequence of the
minus sign appearing in \eqref{eqn:Xtau}.
With this notation we have
\begin{equs}
P_{t - t_{1}}^{1} X^{\bullet}_{\ve} (t_{1}, x) = \begin{tikzpicture}[baseline=5]
		\node[root,label=below:{\tiny$(t,x)$}](l) at (0,0) {};
		\node[bluedot,label=right:{\tiny$t_{1}$}] (m) at (0, 0.3) {};
		\node[noise](h) at (0,0.6){};
		\draw (h) to (m) to (l);
		\end{tikzpicture}, \quad P_{t - t_{1}}^{1}
X^{\tau}_{\ve}(t_{1}, x) = \begin{tikzpicture}[baseline=2]
			\node[bluedot,label=right:{\tiny$t_{1}$}] (0) at (0,0) {};
			\node[root,label=below:{\tiny$(t,x)$}] (l) at (0, -.4) {};
			\draw (l) to (0);
			\draw (0) to (0, 0.3);
			\draw (0,0.3) to (-.5,0.6) node at (-.5,.8) {$\tau_1$} ;
			\draw (0,0.3) to (0,0.6) node at (0,.8) {$\tau_2$} ;
			\draw (0,0.3) to (.5,0.6) node at (.5,.8) {$\tau_{3}$};
		\end{tikzpicture}\, ,
\end{equs} 
for $ \tau = [\tau_{1}, \tau_{2}, \tau_{3}]. $ For convenience we denote by
$ P_{t - t_{1}}^1 \tau_{t_{1}} $ the trees obtained in this manner.
Now we can follow the proof of Proposition \ref{prop:momentX} to write
\begin{equ}
\EE | P_{t - t_{1}}^{1} X^{\tau}_{\ve} (t_{1}, x)|^{2} = \sum_{\gamma} \ [P_{t -
t_{1}}^{1} \tau_{t_{1}}, P_{t - t_{1}}^{1} \overline{\tau}_{t_{1}}]_{\gamma}
(x)\;,
\end{equ}
where $ \overline{\tau} $ is an identical copy of $ \tau $.
The possible paring rules $ \gamma $ for the tree $ P_{t - t_{1}}^{1} \tau_{t_{1}} $ (with an identical copy of itself) are the same as for the tree $ \tau $ (since we only added a
node to the edge entering the root), and also the set of paths $ \{
p_{i} \}_{i=0}^{N-1} $ in which
the tree is decomposed remains unchanged, up to adding the two nodes
corresponding to $ t_{1} $ to the path $ p_{0} $. By the semigroup
property, the contribution of the path $ p_{0} $ is then given by
\begin{equs}
(c e^{t} \varepsilon^{\frac{d}{2} - \alpha})^2(t \vee \varepsilon^2)^{-
\frac{d}{ 2} } \;,
\end{equs}
while the contribution of all other paths remains unchanged, with the
additional constraint that all temporal variables (apart from the red root) are
constrained to the interval $ [0, t_{1}] $.

Hence, following the calculations of the proof of
Proposition~\ref{prop:momentX} we can bound
\begin{equs}
\ [P^{1}_{t - t_{1}}\tau_{t_{1}} , P_{t - t_{1}}^{1} \overline{\tau} ]_{\gamma}
& \leqslant (c e^{t} \varepsilon^{\frac{d}{2} - \alpha})^2(t \vee
\varepsilon^2)^{- \frac{d}{ 2} }\bigg(
\int_{0}^{t_{1}} \mf{s}_{\ve} (s) \ud s \bigg)^{2i(\tau)}\\
& \lesssim (e^t \varepsilon^{\frac d2- \alpha})^{2}(t\vee
\varepsilon^2)^{-\frac d2} \Big( e^{t_{1}} \varepsilon^{1-
\alpha}\Gamma_{d/2}^{1/2}(t_1\varepsilon^{-2})
\Big)^{\l(\tau) -1} \;.
\end{equs} 
This completes the proof of \eqref{est:mPXtau}. For \eqref{est:mPDXtau} one can
follow the same arguments in combination with the estimates in the proof of
Proposition \ref{prop:momentX}.
\end{proof}
Next we can compare the Wild expansion and the original solution via a
comparison principle.
	\begin{lemma}\label{lem:remainder} For every $\ve \in
(0, 1), N\ge1$ and
$(t,x)\in\R_+\times\R^d$, let $ u^{N}_{\ve} $ and $ R^{N} $ be as in
\eqref{eqn:uN}, \eqref{eqn:RN}. Then
		\begin{equation}\label{est:v}
			|u_\varepsilon-u^N_\varepsilon|(t,x)\le P^1*|R^N_\varepsilon|(t,x)\,.
		\end{equation}
	\end{lemma}
	\begin{proof}
		From \eqref{eqn:uep} and \eqref{eqn:uN}, it follows that the remainder $v_\varepsilon=u_\varepsilon-u^N_\varepsilon$ satisfies the following equation
		\begin{equation}\label{eqn:v}
			(\partial_t- \Delta-1)v_\varepsilon=-\(v_\varepsilon^2+3u_\varepsilon^N v_\epsilon+3(u_\varepsilon^N)^2\)v_\varepsilon -R^N_\varepsilon\,,\quad v_\varepsilon(0,\cdot)=0\,.
		\end{equation}
		The estimate \eqref{est:v} is a direct consequence of Lemma \ref{lem:abs.v} and the fact that $v_\varepsilon^2+3u^N_\varepsilon v_\varepsilon+3(u^N_\varepsilon)^2 $ is non-negative.
	\end{proof}
Eventually, if we combine all the results we have obtained so far, we obtain an
estimate on the distance of the Wild expansion from the original solution until
times of order $ (1 - \alpha) \log{\ve^{-1}}.$
	\begin{proposition}\label{prop:uuN} For every $t\le (1- \alpha)\log
\varepsilon^{-1}$, $ q \in [2, \infty) $ and $ N \in \NN $
		\begin{equation}\label{est:uuN}
			\|u_\varepsilon(t,x)-u_\varepsilon^N(t,x)\|_{L^q(\Omega)}\lesssim \varepsilon^{2- 3	\alpha}e^{3t}\(e^t \varepsilon^{1- \alpha}\Gamma_{d/2}^{3/2}(t \varepsilon^{-2})\)^{N-1}
		\end{equation}
		and
		\begin{multline}\label{est:mu}
			\|u_\varepsilon(t,x)\|_{L^q(\Omega)}\lesssim  (t\vee
		\varepsilon^2)^{-\frac d4}e^t \varepsilon^{\frac d2-
		\alpha}\Gamma_{d/2}^{N/2}(t \varepsilon^{-2}) \\
		 + \varepsilon^{2- 3	\alpha}e^{3t}\(e^t \varepsilon^{1- \alpha}\Gamma_{d/2}^{3/2}(t \varepsilon^{-2})\)^{N-1}\,.
		\end{multline}
		In addition, for every $c>1$, uniformly over $t_1(\ve)<t< c\log \varepsilon^{-1}$
		\begin{equation}\label{est:mPu}
			\|P^1_{t-t_1(\ve)}u_\varepsilon(t_1(\ve),y)\|_{L^{q}(\Omega)} \lesssim t^{-\frac d4}e^t \varepsilon^{\frac d2- \alpha}\,.
		\end{equation}
	\end{proposition}
	\begin{proof}
		From the definition of $R^N_\varepsilon$ and the moment estimate in Proposition \ref{prop:momentX}, we see that for every $q\ge2$,
		\begin{align*}
			\|R^N_\varepsilon(t,x)\|_{L^q(\Omega)}
			\lesssim(t\vee \varepsilon^2)^{-3\frac d4}(e^t \varepsilon^{\frac d2- \alpha})^3 
			\sum \(e^t \varepsilon^{1- \alpha}\Gamma_{d/2}^{1/2}(t\varepsilon^{-2}) \)^{\ell(\tau_1)+\ell(\tau_2)+\ell(\tau_3)-3},
		\end{align*}
		where the sum is taken over all trees $\tau_1,\tau_2,\tau_3$ in
$\tree_3^N$ such that $\tau=[\tau_1,\tau_2,\tau_3]\notin\tree_3^N$. Since
$\ell(\tau)=\ell(\tau_1)+\ell(\tau_2)+\ell(\tau_3)\in[N+2,3N]$ and $t$
satisfies $e^t \varepsilon^{1- \alpha}\le 1$, each term in the above sum is at
most $(e^t \varepsilon^{1- \alpha})^{N-1}\Gamma_{d/2}(t
\varepsilon^{-2})^{\frac32(N-1)}$. This yields 
		\begin{equation*}
			\|R^N_\varepsilon(t,x)\|_{L^{q}(\Omega)}\lesssim(t\vee
\varepsilon^2)^{-3\frac d4}(e^t \varepsilon^{\frac d2- \alpha})^3(e^t
\varepsilon^{1- \alpha})^{N-1}\Gamma_{d/2}(t \varepsilon^{-2})^{\frac32(N-1)}\, ,
		\end{equation*}
		which, when combined with Lemma \ref{lem:remainder} gives
		\begin{align*}
			\|u_\varepsilon(t,x)-u^N_\varepsilon(t,x)\|_{L^q(\Omega)}\lesssim(e^t
\varepsilon^{\frac d2- \alpha})^3(e^t \varepsilon^{1-
\alpha}\Gamma_{d/2}^{3/2})^{N-1}\int_0^t(s\vee \varepsilon^2)^{-\frac {3d}4}
\ud s\,.
		\end{align*}
		To derive \eqref{est:uuN}, we observe that $\int_0^t(s\vee
\varepsilon^2)^{-\frac{3d}4} \ud s\lesssim \varepsilon^{2-\frac{3d}2}$ for $d\ge2$. The estimate \eqref{est:mu} is a direct consequence of \eqref{est:uuN} and Proposition \ref{prop:momentX}.

		Let us now show \eqref{est:mPu}. By the triangle inequality, we have
		\begin{equation*}
			\|P^1_{t-t_1}u_\varepsilon(t_1)\|_{L^q(\Omega)}\le
P^1_{t-t_1}\|(u_\varepsilon-u^N_\varepsilon)(t_1)\|_{L^q(\Omega)}+\|P^1_{t-t_1}u^N_\varepsilon(t_1)\|_{L^q(\Omega)}\,.
		\end{equation*}
		The first term is estimated using \eqref{est:uuN}, 
		\begin{align*}
			P^1_{t-t_1}\|( u_\varepsilon-u^N_\varepsilon)
(t_{1})\|_{L^q(\Omega)}(x)
			&\lesssim e^{t-t_1}\varepsilon^{2-3 \alpha}e^{3t_1}
			(e^{t_1}\varepsilon^{1- \alpha}\Gamma^{3/2}_{d/2}(t_1 \varepsilon^{-2}))^{N-1}\,.
		\end{align*}
		Writing the second term as $\sum_{\tau\in\tree^N_3}
P^1_{t-t_1}X^\tau_{\ve}(t_1,x)$ and applying Proposition \ref{prop:mPXtau}, we obtain 
		\begin{equation*}
			\|P^1_{t-t_1}u^N_\varepsilon(t_1,x)\|_{L^q} \lesssim (t\vee \varepsilon^2)^{-\frac d4}e^t \varepsilon^{\frac d2- \alpha}  \,.
		\end{equation*}
		In the previous estimate, we have used the fact that $e^{t_1}\varepsilon^{1- \alpha}\Gamma_{d/2}^{1/2}(t_1 \varepsilon^{-2})\le 1 $ for sufficiently small $\varepsilon$. 
		Combining these estimates yields
		\begin{equation*}
			\|P^1_{t-t_1}u_\varepsilon(t_1,y)\|_q\lesssim t^{-\frac d4}e^t \varepsilon^{\frac d2- \alpha}+e^t\varepsilon^{2-2\bar \alpha- \alpha}\(\varepsilon^{1-\bar \alpha}\Gamma^{3/2}_{d/2}(t_1 \varepsilon^{-2}) \)^{N-1}\,.
		\end{equation*}
		for every $t\ge t_1$. By choosing $N$ sufficiently large, the previous inequality implies \eqref{est:mPu}.
	\end{proof}
\section{Front formation} 
\label{sec:weak_convergence}

\subsection{Tightness criteria}

In this subsection we recall a tightness criterion that is useful in the
upcoming discussion.	
Observe that a subset $\cff$ of $C_\loc(\R^d)$ is compact if and only if for
every compact set $K\subset \R^d$, the projection $\cff_K$ of $\cff$ onto $C(K)$
is compact. In particular the following is a consequence of the
Arzel\`a--Ascoli theorem and of the Morrey--Sobolev inequality.
	\begin{lemma}\label{lem:Xtight}
		Let $\{X_\varepsilon\}_{\varepsilon \in (0, 1)}$ be a family of stochastic processes with sample paths in $C_\loc(\R^d)$. Suppose that for every integer $n\ge1$, there exists $q>d$ such that
		\begin{equation}\label{con:t3}
			\sup_{\varepsilon}\sup_{|z|\le n}\E[ |X_\varepsilon(z)|^q+|\nabla X_\varepsilon(z)|^q]<\infty\,.
		\end{equation}
		Then the family of the laws of $X_\varepsilon$ on $C_\loc(\R^d)$ is tight. 
		In addition, if for every $n\ge1$, there exists a $q>d$ such that
		\begin{equation*}
			\limsup_{\varepsilon\downarrow0}\sup_{|z|\le n}\E [|X_\varepsilon(z)|^q+|\nabla X_\varepsilon(z)|^q]=0\,,
		\end{equation*}
		then $X_\varepsilon$ converges to 0 in probability with respect to the topology of $C_\loc(\R^d)$.
	\end{lemma}
In particular, we deduce the following criterion.

	\begin{lemma}\label{lem:Pg}
		Let $\{g_\varepsilon\}_{\varepsilon \in (0, 1)}$ be a family of random fields on $\R^d$. Let $t(\varepsilon)$ and $L(\varepsilon)$ be positive numbers such that 
$			\lim_{\varepsilon\to0}t(\varepsilon)=\infty$.
		Assume that there exists a $q>d+1$ such that
		\begin{equation*}
			\sup_{\varepsilon}\(1+\frac{L(\varepsilon)}{\sqrt{t(\varepsilon)}}
\) e^{t(\varepsilon)}\sup_{y\in\R^d}\|g_\varepsilon(y)\|_{L^{q}(\Omega)}<\infty\,.
		\end{equation*}
		Then the family $\{P^1_{t(\varepsilon)+t}g_\varepsilon(L(\varepsilon)x):(t,x)\in\R\times\R^d \}_\varepsilon$ is tight on $\Cloc$.
		In addition, if 
		\begin{equation*}
			\limsup_{\varepsilon\to0}\(1+\frac{L(\varepsilon)}{\sqrt{t(\varepsilon)}}
\)e^{t(\varepsilon)}\sup_{y\in\R^d}\|g_\varepsilon(y)\|_{L^{q}(\Omega)}=0\,.
		\end{equation*}
		then $\{P^1_{t(\varepsilon)+t}g_\varepsilon(L(\varepsilon)x):(t,x)\in\R\times\R^d \}$ converges in probability to 0 in $\Cloc$.
	\end{lemma}
	\begin{proof}
		This is a consequence of the following estimates:
		\begin{align*}
			&\|P^1_{t(\varepsilon)+t} g_\varepsilon(L(\varepsilon) x)\|_q\le e^{t(\varepsilon)+t} \sup_y\|g_\varepsilon(y)\|_q\,,
			\\&\|\nabla P^1_{t(\varepsilon)+t} g_\varepsilon(L(\varepsilon)x)\|_q\lesssim L(\varepsilon)(t(\varepsilon)+t)^{-1/2} e^{t(\varepsilon)+t}\sup_y\|g_\varepsilon(y)\|_q\,,
			\\&\|\partial_t P^1_{t(\varepsilon)+t} g_\varepsilon(L(\varepsilon)x)\|_q\lesssim (1+(t(\varepsilon)+t)^{-1}) e^{t(\varepsilon)+t}\sup_y\|g_\varepsilon(y)\|_q\,,
		\end{align*}
		in conjunction with Lemma \ref{lem:Xtight}.
	\end{proof}

\subsection{Weak convergence}

Now recall $ t_{2}(\ve) $ from \eqref{eqn:t-1-2} and define	
	\begin{equation*}
		t_2^\kappa (\ve) \eqdef \Big( \frac d2- \alpha \Big) \log \varepsilon^{-1}-\(\kappa+\frac12\)\log\log \varepsilon^{-1}.
	\end{equation*}
Then we can use the tightness criteria above to describe the behaviour of $
u_{\ve} $ up to times of order $ t_{2}^{\kappa}(\ve) $.

	\begin{lemma}\label{lem:S}
		For every $N>d/(2-2\bar \alpha)-1$ and every $\kappa\ge0$ 
		the process 
		\begin{equation*}
		 (t,x) \mapsto \big(P^1_{t_\star^\kappa(\varepsilon)+t-t_2^\kappa(\varepsilon)}
		 	|u_\varepsilon(t_2^{\kappa}(\varepsilon),\cdot)-P^1_{t_2^\kappa(\varepsilon)-t_1(\varepsilon)}u^N_\varepsilon(t_1(\varepsilon),\cdot)|\big)(x
L_{\ve})\;,
		\end{equation*} converges to 0 in probability on $\Cloc$.
	\end{lemma}
	\begin{proof}
		For each $t\ge t_1$, we define
		\begin{align*}
			S^N_\varepsilon(t,x)=u_\varepsilon(t,x)-P^1_{t-t_1}u^N_\varepsilon(t_1,x)\,.
		\end{align*}
		Since $h_\varepsilon(t,x) \eqdef P^1_{t-t_1}u^N_\varepsilon(t_1,x)$ solves $(\partial_t- \Delta-1)h=0$ with initial condition (at time $t_1$) $h(t_1,x)=u^N_\varepsilon(t_1,x)$, the difference $S^N_\varepsilon=u_\varepsilon-h_\varepsilon$ satisfies
		\begin{equation*}
			(\partial_t- \Delta-1)S^N_\varepsilon=-\((S^N_\varepsilon)^2+3(S^N_\varepsilon) h_\varepsilon+3h_\varepsilon^2\)S^N_\varepsilon-h_\varepsilon^3 \quad\mbox{for}\quad t>t_1
		\end{equation*}
		with initial condition $S^N_\varepsilon(t_1,x)=u_\varepsilon(t_1,x)-u^N_\varepsilon(t_1,x) $.
		An application of Lemma~\ref{lem:abs.v} gives
		\begin{equation*}
			|S^N_\varepsilon(t_2^\kappa,x)|\le
\int_{t_1}^{t_2^\kappa}\int_{\R^d}
p^1_{t_2^\kappa-s}(x-y)|P^1_{s-t_1}u_\varepsilon^{N}(t_1,y)|^3 \ud y \ud s
			+P^1_{t_2^\kappa-t_1}|u_\varepsilon-u^N_\varepsilon|(t_1,x) \,.
		\end{equation*}
		Let $q>d+1$ be fixed. Applying Proposition~\ref{prop:uuN}, we have 
		\begin{align*}
			\sup_y\|S^N_\varepsilon(t_2^\kappa,y)\|_q
			&\lesssim 
			\int_{t_1}^{t_2^\kappa}e^{t_2^\kappa-s}\(e^{s}
s^{-\frac d4}\varepsilon^{\frac d2- \alpha}\)^3 \ud s
			+ e^{t_2^\kappa-t_1} \varepsilon^{2-3\bar \alpha}\(\varepsilon^{1-\bar \alpha}\Gamma_{d/2}^{3/2}(t_1 \varepsilon^{-2}) \)^{N-1}
		\end{align*}
		and hence 
		\begin{equs}
			e^{t_\star^\kappa-t_2^\kappa}\sup_y\|S^N_\varepsilon(t_2^\kappa,y)\|_{q}
			&\lesssim  e^{t_\star^\kappa-t_2^\kappa}
\(e^{t_2^\kappa}\varepsilon^{\frac d2 -\alpha}t_1^{-\frac d4}\)^3 (t_2^\kappa -t_1) \\
&\quad +e^{t_\star^\kappa-t_1}\varepsilon^{2-3\bar \alpha}\(\varepsilon^{1-\bar \alpha}\Gamma_{d/2}^{3/2}(t_1 \varepsilon^{-2}) \)^{N-1}\,.
		\end{equs}
		From the definitions it is evident that
		\begin{align*}
			e^{t_\star^\kappa-t_2^\kappa}
\(e^{t_2^\kappa}\varepsilon^{\frac d2 -\alpha}t_1^{-\frac d4}\)^3(t_2^\kappa-t_1)
			\lesssim \(\log \varepsilon^{-1}\)^{-\frac d2- \kappa}.
		\end{align*}
We observe that at this point we used that $ t_{2}(\ve) $ contains an
additional negative term $ - \frac{1}{2} \log{\log{\ve^{-1}}} $, at least in
dimension $ d =2 $, to obtain an upper bound that vanishes as $ \ve \to 0 $.
		As for the second term, we have
		\begin{align*}
			e^{t_\star^\kappa-t_1}\varepsilon^{2-3\bar \alpha}\(\varepsilon^{1-\bar \alpha}\Gamma_{d/2}^{3/2}(t_1 \varepsilon^{-2}) \)^{N-1}
			\lesssim \varepsilon^{(1-\bar \alpha)(N+1)-\frac d2}\(\log \varepsilon^{-1}\)^{-\kappa}\Gamma_{d/2}^{3(N-1)/2}(t_1 \varepsilon^{-2})\,,
		\end{align*}
		which also vanishes as $\varepsilon\to0$ by our choice of $N$.
		It follows that
		\begin{equation*}
			\limsup_{\varepsilon\to0}\(1+\frac{L_\varepsilon}{\sqrt{t_\star^\kappa-t_2^\kappa}} \) e^{t_\star^\kappa-t_2^\kappa}\sup_y\|S^N_\varepsilon(t_2^\kappa,y)\|_{q}=0\,.
		\end{equation*}
		Applying the second part of Lemma \ref{lem:Pg}, the above
estimates imply that the process $(t,x)\mapsto
P^1_{t_\star^\kappa-t_2^\kappa+t}S^N_\varepsilon(t_2^\kappa,L_\varepsilon x)$
converges to 0 in probability in $C_\loc(\R\times\R^d)$ .
	\end{proof}
Next, we observe that the truncated Wild expansion converges to the Bargmann--Fock
field around time $ t_{\star} (\ve) $.

	\begin{lemma}\label{lem:Pu1convg}
		For every $\kappa\ge0$, as $\varepsilon\to0$, we have:
		\begin{enumerate}[(i)]
			\item The process $\{e^{-\kappa\log\log
\varepsilon^{-1}} P^1_{t_\star^\kappa-t_1+t}X^\bullet_{\ve}(t_1,L_\varepsilon
x):(t,x)\in\R\times\R^d \}$ converges in law in $C_{\loc}(\R\times\R^d)$ to
$\{e^t \Psi(x):(t,x)\in\R\times\R^d \} $ with $ \Psi $ as in \eqref{eqn:def-Psi}.
			\item For each $\tau\in\tree_3\setminus\{\bullet\}$,
the process $\{e^{-\kappa\log\log
\varepsilon^{-1}}P^1_{t_\star^\kappa-t_1+t}X^\tau_{\ve}(t_1,L_\varepsilon x):(t,x)\in\R\times\R^d \}$ converges in probability to $0$ in $C_{\loc}(\R\times\R^d)$.
		\end{enumerate}
	\end{lemma}
	\begin{proof}
		We first show that the family $\{e^{-\kappa\log\log
\varepsilon^{-1}}P^1_{t_\star^\kappa-t_1+t}X^\tau_{\ve}(t_1,L_\varepsilon
x):(t,x)\in\R\times\R^d \}_\varepsilon$ is tight in $\Cloc$ for every $\tau\in\tree_3$.
		Hence, let us fix $q>d+1$. From Proposition \ref{prop:uuN}, we have
		\begin{align*}
			e^{t_\star^\kappa-t_1+t}e^{-\kappa\log\log
\varepsilon^{-1}}\sup_y\|X^\tau_{\ve}(t_1,y)\|_{L^q}\lesssim
e^{t_\star+t}\varepsilon^{\frac d2- \alpha}t_1^{-\frac d4}=\coo(1)\quad
\textrm{as}\quad\varepsilon\downarrow 0\,.
		\end{align*}
		In addition, we observe that
		\begin{equation*}
			\frac{L_\varepsilon}{\sqrt{t_\star^\kappa-t_1}}=\coo(1)
			\quad\textrm{as} \quad \varepsilon \downarrow0\,,
		\end{equation*}
		so that an application of Lemma \ref{lem:Pg} implies tightness
for the sequence.

		Let us now prove the first point. By construction, the process
$$(t,x)\mapsto e^{-\kappa\log\log
\varepsilon^{-1}}P^1_{t_\star-t_1+t}X^\bullet_{\ve}(t_1,L_\varepsilon
x)=e^{-\kappa\log\log \varepsilon^{-1}}
X^\bullet_{\ve}(t_\star^\kappa+t,L_\varepsilon x)$$ is a centered Gaussian
random field with covariance
		\begin{multline*}
			\E[e^{-\kappa\log\log
\varepsilon^{-1}}X^\bullet_{\ve}(t_\star^{\kappa}+t,L_\varepsilon
x)e^{-\kappa\log\log \varepsilon^{-1}}X^\bullet_{\ve}(t_\star^{\kappa}+t,L_\varepsilon y) ]
			\\=e^{2 t}e^{2t_\star}\varepsilon^{d-2\alpha}p^\varepsilon_{t_\star^\kappa+t}*p^\varepsilon_{t_\star^\kappa+t}(L_\varepsilon(x-y))\,.
		\end{multline*}
		Now, it is straightforward to verify that 
		\begin{equ}
			\lim_{\varepsilon\to0} \ (t_\star^{\kappa})^{-\frac d2}
e^{2t_\star}\varepsilon^{d-2 \alpha}=\(\frac d2- \alpha\)^{-\frac d2} e^{2 \mf{c}}
		\end{equ}
		with $ \mf{c} $ as in \eqref{eqn:t-star} and
		\begin{equation*}
			\lim_{\varepsilon\to0} \ (t_\star^{\kappa})^{\frac d2}
p^\varepsilon_{t_\star^\kappa+t}*p^\varepsilon_{t_\star^\kappa+t}(L_\varepsilon(x-y))=(8
\pi)^{-\frac d2}\exp\lt(-\frac{|x-y|^2}{8} \rt)\,,
		\end{equation*}
		so that (i) is verified.

		As for the second point, let $\tau\in\tree_3\setminus\{\bullet\}$ and $q$ be fixed such that $q>d+1$. From Proposition \ref{prop:mPXtau}, we have
		\begin{equation*}
			e^{t_\star^\kappa-t_1}e^{-\kappa\log\log \varepsilon^{-1}}\sup_{y}\|X^\tau_{\ve}(t_1,y)\|_q\lesssim \(\varepsilon^{1-\bar \alpha}\Gamma_{d/2}^{1/2}(t_1 \varepsilon^{-2})\)^{\ell(\tau)-1}\,.
		\end{equation*}
		Since $\bar \alpha<1$ and $\ell(\tau)>1$, the right-hand side
above vanishes as $\varepsilon\to0$. By Lemma \ref{lem:Pg}, this implies (ii).
	\end{proof}

	\begin{remark}
		It is possible to show that the process
$\{P^1_{t_\star-t_1+t}u_\varepsilon(t_1,L_\varepsilon
x):(t,x)\in\R\times\R^d \}$ converges in law to $\{e^t \Psi(x):(
t,x)\in\R\times\R^d \} $ in
$C_{\loc}(\R\times\R^d)$. However, this fact is not needed in what follows, so its proof is omitted.
	\end{remark}
	In the upcoming result we will work with the flow $
\overline{\Phi}(t,u) $ given by
	\begin{equation}\label{eqn:Phi-bar}
		\overline{\Phi}(t,u)=\frac{e^tu}{\(1+(e^{2t}-1)u^2
\)^{1/2}}\;,\qquad t\in \R\;,\quad u \in \RR\;,
	\end{equation}
which solves \eqref{eqn:Phi-new} with initial condition $
\overline{\Phi}(0, u) = u $ (the  feature that
distinguishes $ \Phi $ from $ \overline{\Phi} $ is the initial condition).
With this definition we have for $ t > 0 $  
	\begin{equs}[eqn:bounds-on-Phi]
		\partial_u \overline{\Phi}(t,u)&=e^t(1+u^2(e^{2t}-1))^{-\frac32}\,,
		\\\frac{\partial_u^2 \overline{\Phi}(t,u)}{\partial_u \overline{\Phi}(t,u)}&= 3u\frac{e^{2t}-1}{1+u^2(e^{2t}-1)}
		\quad\mbox{and}\quad|\partial_u^2 \overline{\Phi}(t,u)|\lesssim e^{2t} \,.
	\end{equs}
To obtain the last bound we observe that
\begin{equs}
| \partial_{u}^{2} \overline{\Phi}(t, u) | & = \frac{3 |u | e^{t}(e^{2t}-1)}{(1 +
  u^{2}(e^{2 t} -1))^{\frac{5}{2}}}.
\end{equs}
Then for $ |u| \leqslant  e^{- t} $ we have
$| \partial_{u}^{2} \overline{\Phi}(t, u)| \lesssim e^{2t} -1 \lesssim
e^{2 t},$
since in the denominator
$1 +  u^{2}(e^{2 t} -1) \geqslant 1$.
Instead for $ |u| > e^{-t} $ the leading term in the denominator
is $ u^{2}e^{2 t},$ so we find
\begin{align*}
  | \partial_{u}^{2} \Phi(t, u) | & = \frac{3 |u | e^{t}(e^{2t}-1)}{(1 +
  u^{2}(e^{2 t} -1))^{\frac{5}{2}}} \leqslant  \frac{3 |u | e^{t}(e^{2t}-1)}{
  u^{5}e^{5 t}} \lesssim \frac{1}{| u |^{4}} \frac{e^{3t}}{e^{5t}}  \lesssim e^{ 2t}.
\end{align*}
	Next, let us define for every $t>t_2^\kappa$
	\begin{align*}
		w^{N,\kappa}_\varepsilon(t,x)
		&=\overline{\Phi}(t-t_2^\kappa,P_{t-t_2^\kappa}P^1_{t_2^\kappa-t_1} u^N_\varepsilon(t_1,x))\\
&=\overline{\Phi}(t-t_2^\kappa,e^{-(t-t_2^\kappa)}P^1_{t-t_1} u^N_\varepsilon(t_1,x))\,.
	\end{align*}
	When $\kappa=0$, we simply write $w^{N}_{\ve}$ for $w^{N,0}_{\ve}$. Now
we can prove the main result of this section, which states that $ w^{N, \kappa}_{\ve} $ is a good approximation of $
u_{\ve} (t_{\star}^{\kappa}(\ve), x L_{\ve}) $. 
	\begin{proposition}\label{prop:d3}
		For every $\kappa\in[0,\frac14)$, the process 
		\begin{equation*}
		 	(t,x) \mapsto |u_\varepsilon(t_\star^\kappa(\ve)+t,x L_{\ve}
)-w^{N,\kappa}_\varepsilon(t_\star^\kappa(\ve) +t,x L_{\ve})|\;,
		\end{equation*} 
converges to $0$ in probability in $\Cloc$ as $\varepsilon\to0$.
	\end{proposition}
	\begin{proof}		
		From its definition, we see that $w^{N,\kappa}_\varepsilon$
satisfies for $ t > t_{2}^{\kappa}(\ve) $
		\begin{equation*}
			(\partial_t-
\Delta-1)w^{N,\kappa}_\varepsilon=-(w^{N,\kappa}_\varepsilon)^3+f^N_\varepsilon\,,
\quad
w^{N,\kappa}_\varepsilon(t_2^\kappa,\cdot)=P^1_{t_2^\kappa-t_1}u^N_\varepsilon(t_1,\cdot)\,,
		\end{equation*}
		where
		\begin{equation*}
			f^N_\varepsilon(t,x)=-\partial_u^2 \overline{\Phi}\(t-t_2^\kappa,e^{-(t-t_2^\kappa)}P^1_{t-t_1}u^N_\varepsilon(t_1,x)\)e^{-2(t-t_2^\kappa)} |\nabla P^1_{t-t_1}u^N_\varepsilon(t_1,x)|^2\,.
		\end{equation*}
		It follows that the difference $v^N_\varepsilon=w^{N,\kappa}_\varepsilon-u_\varepsilon$ satisfies the equation
		\begin{equation*}
			(\partial_t- \Delta-1)v^N_\varepsilon=-((v^N_\varepsilon)^2+3u_\varepsilon v^N_\varepsilon+3u_\varepsilon^2 )v^N_\varepsilon+f^N_\varepsilon\,\quad\forall t>t_2^\kappa
		\end{equation*}
		with initial condition at $t_2^\kappa$ given by $v^N_\varepsilon(t_2^\kappa,\cdot)=P^1_{t_2^\kappa-t_1}u^N_\varepsilon(t_1,\cdot)-u_\varepsilon(t_2^\kappa,\cdot)$.
		Applying Lemma \ref{lem:abs.v}, we see that
		\begin{align*}
			|u_\varepsilon-w^{N,\kappa}_\varepsilon|(t,x)
			&\le P^1_{t-t_2^\kappa}|P^1_{t_2^\kappa-t_1}u^N_\varepsilon(t_1,\cdot)-u_\varepsilon(t_2^\kappa,\cdot) |(x)
			\\&\quad+
\int_{t_2^\kappa}^t\int_{\R^d}p^1_{t-s}(x-y)|f^N_\varepsilon(s,y)| \ud y \ud s
		\end{align*}
		for every $(t,x)\in[t_2^\kappa,\infty)\times\R^d$. 
		By Lemma \ref{lem:S}, the process
$$(t,x) \mapsto \big(P^1_{t_\star^\kappa+t-t_2^\kappa}|P^1_{t_2^\kappa-t_1}u^N(t_1,\cdot)-u_\varepsilon(t_2^\kappa,\cdot)
|\big)(L_\varepsilon x)$$ converges to 0 in probability
in $\Cloc$.
		We also note that in view of \eqref{eqn:bounds-on-Phi}, $|f^N_\varepsilon(t,x)|\lesssim |\nabla P^1_{t-t_1}u^N_\varepsilon(t_1,x)|^2$.
		Hence it remains to show that the process 
		\begin{equation*}
		 	H_\varepsilon(t,x) = \int_{t_2^\kappa}^{t_\star^\kappa+t}\int_{\R^d}p^1_{t_\star^\kappa+t-s}(L_\varepsilon
x-y)|\nabla P^1_{s-t_1}u^N_\varepsilon(t_1,y)|^2 \ud y \ud
s
		\end{equation*} 
		 converges to $0$ in probability in $\Cloc$.
		
		For every $z_1=(\sigma_1,x_1),z_1=(\sigma_2,x_2)$ in $\R\times\R^d$, define the distance $d(z_1,z_2)=|\sigma_1- \sigma_2|^{\frac12}+|z_1-z_2|$. Let $K$ be a compact set in $\R^{d+1}$ and $q=d+2$. It suffices to show that
		\begin{equation}\label{WTS:KK}
			\lim_{\varepsilon\to0}\sup_{z\in K}\E|H_\varepsilon(z)|^q+
			\E\int_{K\times
K}\frac{|H_\varepsilon(z_1)-H_\varepsilon(z_2) |^q }{d(z_1,z_2)^q } \ud z_1 \ud z_2=0\,.
		\end{equation}
		In fact, by the Morrey--Sobolev inequality, the above estimate implies that  $\sup_{z\in K}|H_\varepsilon(z)|$ has vanishing $q$-th moment as $\varepsilon\to0$, which in turn, implies the convergence of $H_\varepsilon$ to $0$ in $C_\loc(\R\times\R^d)$ in probability. 
		From Proposition \ref{prop:mPXtau}, we have
		\begin{equation*}
			\|\nabla P^1_{s-t_1}u^N_\varepsilon(t_1,y)\|_q\lesssim s^{-\frac d4-\frac12}e^s \varepsilon^{\frac d2- \alpha }\,,
		\end{equation*}
		which implies that
		\begin{align}
			e^{t_\star^\kappa-s} \sup_{y\in\R^d}\|\nabla P^1_{s-t_1}u^N_\varepsilon(t_1,y)\|_q^2
			&\lesssim \((t_2^\kappa)^{-\frac d4-\frac12}e^{t_\star^\kappa}\varepsilon^{\frac d2- \alpha} \)^2
			\nonumber\\&\lesssim\(\log \varepsilon^{-1}\)^{2 \kappa-1}\(\frac{\log \varepsilon^{-1}}{t_2^\kappa} \)^{\frac d2+1} \,,
			\label{tmp:stt}
		\end{align}
		for every $s\in(t_2^\kappa,t_\star^\kappa+\dist(0,K))$.
		It follows that for every $z=(\sigma,x)$ in $K$ 
		\begin{align*}
			\|H_\varepsilon(\sigma,x)\|_q
			&\lesssim
\int_{t_2^\kappa}^{t_\star^\kappa+\sigma}e^{t_\star^\kappa-s} \sup_{y\in\R^d}\|\nabla
P^1_{s-t_1}u^N_\varepsilon(t_1,y)\|_{2q}^2 \ud s
			\\&\lesssim \(\log \varepsilon^{-1}\)^{2 \kappa-1}\(\frac{\log \varepsilon^{-1}}{t_2^\kappa} \)^{\frac d2+1} \(t_\star^\kappa-t_2^\kappa+\dist(0,K)\)\,.
		\end{align*}
		From the definitions of $t_2$ and $t_\star$, the right-hand side above vanishes as $\varepsilon\to0$.

		 For every $z_1=(\sigma_1,x_1),z_2=(\sigma_2,x_2)\in K$, we now estimate the increment $H_\varepsilon(z_2)-H_\varepsilon(z_1) $. The increment in the spatial variables $\|H_\varepsilon(\sigma_1,x_1)-H_\varepsilon(\sigma_1,x_2)\|_q$ can be estimated uniformly by a constant multiple of
		\begin{align*}
			|x_1-x_2|\int_{t_2}^{t_\star+\sigma_1}\int_{\R^d}L_\varepsilon|\nabla
p^1_{t_\star+\sigma_1-s}(y)| \ud y \bigg( \sup_{y} \|\nabla
P^1_{s-t_1}u^N_\varepsilon(t_1,y)\|_{2q}^2\bigg) \ud s\,.
		\end{align*}
		Taking into account the fact that $L_\varepsilon\|\nabla p^1_{t_\star+\sigma_1-s}\|_{L^1(\R^d)}\lesssim L_\varepsilon (t_\star+\sigma_1-s)^{-\frac12} e^{t_\star-s} $ uniformly for every $z_1\in K$, this gives the estimate
		\begin{multline*}
			\|H_\varepsilon(\sigma_1,x_1)-H_\varepsilon(\sigma_1,x_2)\|_q
			\\\lesssim
|x_1-x_2|\int_{t_2}^{t_\star+\sigma_1}L_\varepsilon(t_\star+ \sigma_1-s)^{-\frac12}
\sup_{y} e^{t_\star-s} \|\nabla P^1_{s-t_1}u^N_\varepsilon(t_1,y)\|_{2q}^2 \ud s\,.
		\end{multline*}
		Simplifying the integration on the right-hand side, using the estimate \eqref{tmp:stt}, we obtain
		\begin{multline*}
			\|H_\varepsilon(\sigma_1,x_1)-H_\varepsilon(\sigma_1,x_2)\|_q
			\\\lesssim |x_1-x_2|\(\log \varepsilon^{-1}\)^{2 \kappa-\frac12}\(\frac{\log \varepsilon^{-1}}{t_2^\kappa} \)^{\frac d2+1} \(t_\star^\kappa-t_2^\kappa+\sigma\)^{\frac12}\,.
		\end{multline*}
		To estimate the increment in the time variables, we assume without lost of generality that $\sigma_1<\sigma_2$ and write
		\begin{align*}
			&H(\sigma_2,x_2)-H(\sigma_1,x_2)
			\\&=\int_{t_\star^\kappa+\sigma_1}^{t_\star^\kappa+\sigma_2}\int_{\R^d}p^1_{t_\star^\kappa+\sigma_2-s}(L_\varepsilon
x-y)|\nabla P^1_{s-t_1}u^N_\varepsilon(t_1,y)|^2 \ud y \ud s
			\\&\quad+\int_{t_2^\kappa}^{t_\star^\kappa+\sigma_1}\int_{\R^d}\(p^1_{t_\star^\kappa+\sigma_2-s}-p^1_{t_\star^\kappa+\sigma_1-s}
\)(L_\varepsilon x-y)|\nabla P^1_{s-t_1}u^N_\varepsilon(t_1,y)|^2 \ud y \ud s
			\\&=:I_1+I_2\,.
		\end{align*}
		The first term is estimated easily
		\begin{align*}
			\|I_1\|_q&\lesssim
\int_{t_\star^\kappa+\sigma_1}^{t_\star^\kappa+\sigma_2}e^{t_\star^\kappa+\sigma_2-s}
\sup_y \|\nabla P^1_{s-t_1}u^N_\varepsilon(t_1,y)\|_{2q}^2 \ud s
			\\&\lesssim |\sigma_2- \sigma_1|\(\log \varepsilon^{-1}\)^{2 \kappa-1}\(\frac{\log \varepsilon^{-1}}{t_2^\kappa} \)^{\frac d2+1}\,.
		\end{align*}
		For the second term, we use the elementary estimates
		\begin{align*}
			|p^1_{t_\star^\kappa+\sigma_2-s}(y)-p^1_{t_\star^\kappa+\sigma_1-s}(y)|
			\le\int_{\sigma_1}^{\sigma_2}|\partial_\sigma
p^1_{t_\star^\kappa+\sigma-s}(y)| \ud \sigma
		\end{align*}
		and 
		\begin{align*}
			\|\partial_\sigma p^1_{t_\star^\kappa+\sigma-s}\|_{L^1(\R^d)}
			\lesssim e^{t_\star^\kappa+\sigma-s}\(1+(t_\star^\kappa+\sigma-s)^{-1} \)
		\end{align*}
		to obtain
		\begin{align*}
			\|I_2\|_q\lesssim
\int_{t_2^\kappa}^{t_\star^\kappa+\sigma_1}\int_{\sigma_1}^{\sigma_2}\(1+(t_\star^\kappa+\sigma-s)^{-1}
\)d \sigma e^{t_\star^\kappa+\sigma_2-s}\sup_{y} \|\nabla
P^1_{s-t_1}u^N_\varepsilon(t_1,y)\|_{2q}^2 \ud s\,.
		\end{align*}
		Using the estimate \eqref{tmp:stt}, it is straightforward to verify that
		\begin{align*}
			\|I_2\|_q
			\lesssim |\sigma_2- \sigma_1|^{\frac12}\(\log \varepsilon^{-1}\)^{2 \kappa-1}\(\frac{\log \varepsilon^{-1}}{t_2^\kappa} \)^{\frac d2+1}(t_\star^\kappa-t_2^\kappa+\dist(0,K))^{\frac12}\,.
		\end{align*}
		Combining these estimates yields
		\begin{multline*}
			\E\int_{K\times
K}\frac{|H_\varepsilon(z_1)-H_\varepsilon(z_2) |^q }{d(z_1,z_2)^q } \ud z_1 \ud z_2
			\\\lesssim \(\log \varepsilon^{-1}\)^{2 \kappa-\frac12}\(\frac{\log \varepsilon^{-1}}{t_2^\kappa} \)^{\frac d2+1} \(t_\star^\kappa-t_2^\kappa+\dist(0,K)\)^{\frac12}\,.
		\end{multline*}
		Since $\kappa<\frac14$, the right-hand side above vanishes as $\varepsilon\to0$, 
		which implies \eqref{WTS:KK} and completes the proof.
	\end{proof}


\section{Front propagation}\label{sec:convergence-mcf}

In this section we prove Proposition~\ref{prop:front-formation}
regarding the formation of the initial front and Theorem~\ref{thm:mcf} regarding its
evolution via mean curvature flow. We start by recalling an \textit{a-priori}
bound on solutions to the Allen--Cahn equation.
\begin{lemma}\label{lem:cmng-dwn-inf}
For every $ u_{0} \in C_{\mathrm{loc}}(\RR^{d}) $ satisfying $ \sup_{x \in
\RR^{d}} | u_{0}(x) e^{- |x|} | < \infty$, let $ u $ be the solution to
the Allen--Cahn equation \eqref{eqn:uep} with initial condition $
u_{0} $. Then
\begin{equs}
| u(t) | \leqslant \frac{e^{t}}{\sqrt{ e^{2 t} - 1}}, \qquad \forall t > 0 \;.
\end{equs}
\end{lemma}

\begin{proof}
We observe that with the definition of $ \overline{\Phi} $ as in
\eqref{eqn:Phi-bar} we have $ \Xi(t) = e^{t}/ \sqrt{e^{2t} -1} = \lim_{u \to
\infty} \overline{\Phi} (t, u)$, so that $ \Xi(t) $ is a space-independent solution to the
Allen--Cahn equation on $ (0, \infty) $. By comparison, since $ \lim_{t \to 0}
\Xi(t) = \infty $, and approximating $ u_{0} $ with uniformly bounded initial
conditions, we obtain $ u(t) \leqslant \Xi (t)  $ for all $ t > 0 $.
Similarly one derives the upper bound. The growth condition on $ u_{0} $ is
used to justify the approximation procedure, as well as the well-posedness of
the heat flow started in $ u_{0} $.
\end{proof}
Since our results are concerned with 
convergence in law of the process $ u_{\ve} $, the choice of
underlying probability space is not relevant. In the next lemma we build a
probability space on which our family of processes converges in probability. We
note that the Skorokhod representation theorem is usually stated for discrete
families of random variables, rather than continuous ones hence, for
completeness, we include a proof of our statement.

\begin{lemma}\label{lem:prob-space}
There exists a probability space $ (\Omega, \mF, \PP) $ supporting a sequence of processes $ \{u_{\ve}(t,x)  \ \colon \  (t,x) \in
[t_{\star}(\ve), \infty) \times \RR^{d} \} $ with the same law as the solutions to
\eqref{eqn:uep}, and a Gaussian process $ \{ \Psi(x)  \ \colon \ x \in
\RR^{d} \} $ satisfying \eqref{eqn:def-Psi}, 
such that $ \{ u_{\ve}(t_{\star}(\ve), x L_{\ve})  \ \colon \
x \in \RR^{d} \} $ converges, as $ \ve \to 0 $, to $ \{ \Phi(0, \Psi (x))  \
\colon \ x \in \RR^{d} \} $ in probability in $ C_{\mathrm{loc}}(\RR^{d})$.
\end{lemma}

\begin{proof}
Let us consider a probability space $ (\Omega, \mF, \PP) $ supporting a white
noise $ \eta $ on $ \RR^{d} $. Then define, for $ \ve \in (0, 1) $, the random
fields $ x \mapsto X_{\ve}^{\bullet}(t_\star(\ve), x L_{\ve}) $ by
\begin{equs}
X_{\ve}^{\bullet}(t_{\star} (\ve), x L_{\ve}) & = K_{\ve} * \varphi^{\ve} * \eta, \\
K_{\ve}(x) & = e^{t_{\star}(\ve) + \mf{c}} \ve^{\frac{d}{2} - \alpha}
L_{\ve}^{\frac{d}{2}} (4 \pi t_{\star}(\ve) )^{- \frac{d}{2} } \exp \Big\{ -
\frac{| x |^{2} L_{\ve}^{2} }{4 t_{\star}(\ve)} \Big\} \;.
\end{equs}
Here $ \varphi^{\ve} $ is as in \eqref{e:defEtax}, $ \mf{c} $ as in \eqref{eqn:t-star} and we observe that in so far $ x \mapsto  X_{\ve}^{\bullet} (t_{\star}(\ve), x L_{\ve}) $ is a
time-independent Gaussian process constructed to have the same law as the
solution to \eqref{eqn:Xtrunk}
at time $ t_{\star}(\ve) $. From the definition of $ t_{\star}(\ve) $ we obtain
the convergence of $ K_{\ve}(x) \to K(x) $, with
\begin{equs}
K(x)  = \frac{(8 \pi)^{\frac{d}{4}
}}{(4 \pi)^{\frac{d}{2}}} \exp \Big( - \frac{| x |^{2}}{4} \Big) \;,
\end{equs}
from which we obtain that $ X_{\ve}^{\bullet} (t_{\star}(\ve), x L_{\ve}) \to
\Psi(x) $ uniformly over $ x $, almost surely (with $ \Psi $ having the
required covariance structure), as $ \ve \to 0 $. In addition, starting from
the process $ X_{\ve}^{\bullet} (t_{\star}(\ve), \cdot) $ we can construct
a sequence of noises $ \eta_{\ve} $ with the same law as (but not
identical to) the initial conditions $ \ve^{\frac{d}{2} - \alpha} \eta *
\varphi^{\ve} $ appearing in \eqref{eqn:uep}, and such that
$ X_{\ve}^{\bullet}(t_{\star}(\ve), x) = P^{1}_{t_{\star}(\ve)}
 \eta_{\ve} $.

Now we can follow step by step the proof of
Theorem~\ref{thm:convergvence-Gaussian}, to find that if we consider $
u_{\ve} $ the solution to \eqref{eqn:uep} with the initial condition $
\eta_{\ve} $ we just constructed, then $ u_{\ve}(t_{\star}(\ve), x
L_{\ve}) \to \Phi(0, \Psi(x)) $ in probability in $
C_{\mathrm{loc}}(\RR^{d}) $ as $ \ve \to 0 $.
\end{proof}
The next lemma establishes the formation of
the fronts by time $ t_{\star}^{\kappa}(\ve) $ for some $ \kappa \in(0,
\frac{1}{2}]$. We write
$C^{k}_{b}(\RR^{d}; \RR)$
for the space of $ k$ times differentiable functions with all derivatives continuous and
bounded. 
Recall further that \[ t_{\star}^{\kappa}(\ve) = t_{\star}(\ve) +
2 \kappa \log{ L_{\ve} } - 2 \kappa \log{(d/2 - \alpha)} \,. \] 
We the define the random nodal set
\begin{equs}
\Gamma_{1} = \{  x \in \RR^{d}  \ \colon \ \Psi(x) = 0\} \;,
\end{equs}
and recall the definition of $ K^{1}_{\delta} $ from \eqref{eqn:K-sets}.
The proof of the following lemma and of the subsequent proposition
follow roughly the approach of \cite[Theorem 4.1]{barles1998}.

\begin{lemma}\label{lem:front-formation-loglog}
Consider $ (\Omega, \mF, \PP) $ as in Lemma~\ref{lem:prob-space}. For any $ 0 < \kappa \leqslant \frac{1}{2} $ and any sequence
$ \{t(\ve)\}_{\ve \in (0, 1)} $ with $ t(\ve) \geqslant t_{\star}(\ve) $ such that
\begin{equ}
\limsup_{\ve \to 0} \big( t(\ve) - t^{\frac{1}{2}}_{\star}(\ve) \big)  \leqslant 0
\leqslant  \liminf_{\ve \to 0} \big( t(\ve) - t^{\kappa}_{\star}(\ve)\big) \; ,
\end{equ}
it holds that for all $ \delta, \zeta \in (0, 1) $ 
\begin{equs}
\lim_{\ve \to 0} \PP \( \| u_{\ve}(t(\ve), \cdot L_{\ve}) - \sgn (\Psi
(\cdot)) \|_{K_{\delta}^{1}} > \zeta\) = 0 \; .
\end{equs}
\end{lemma}

\begin{proof}
Suppose that along a subsequence $ \{ \ve_{n} \}_{n \in \NN} $, with $
\ve_{n} \in (0, 1), \lim_{n \to \infty} \ve_{n} =0
$, it holds that for
some $ \delta, \zeta , \zeta^{\prime}  \in (0, 1) $
$$ \lim_{n \to \infty} \PP (  \| u_{\ve_{n}}(t(\ve_{n}), \cdot L_{\ve_{n}}) - \sgn (\Psi
(\cdot)) \|_{K_{\delta}^{1}} > \zeta ) \geqslant \zeta^\prime.$$ 
By our choice of probability space, up to further refining the subsequence, we can assume that $
u_{\ve_{n}}(t_{\star} (\ve_{n}), \cdot L_{\ve_{n}}) \to \Phi(0, \Psi(\cdot)) $ almost
surely in $ C_{\mathrm{loc}}(\RR^{d}) $. We will then show that our
assumption is absurd, by proving that almost surely
\begin{equs}
\lim_{n \to \infty}   \| u_{\ve_{n}}(t(\ve_{n}), \cdot L_{\ve_{n}}) - \sgn (\Psi
(\cdot)) \|_{K_{\delta}^{1}} = 0 \; .
\end{equs}
In particular, it will suffice to show that for any $ x_{0} \in \Gamma_{1}^{c} $ there exist (random) $
\lambda(x_{0}), \varrho(x_{0}), \ve (x_{0}) > 0 $ such that for all $
\ve_{n} \in (0, \ve (x_{0})) $:
\begin{equ}[eqn:prf-slowmcf-1]
 1 - \frac{\lambda}{L_{\ve_{n}}^{2 \kappa}} \leqslant \sgn( \Psi (x)) \cdot
u_{\ve_{n}}(t(\ve_{n}), x
L_{\ve_{n}}) \leqslant 1 + \lambda \ve^{\frac{d}{2} - \alpha}_{n}, \quad \forall x \in B_{\delta}(x_{0}).
\end{equ}
For the sake of clarity, let us refrain from writing the subindex $ n $ and fix
$ x_{0} $ such that $ \Psi(x_{0}) > 0 $ (the opposite case follows analogously).
For the upper bound we use Lemma~\ref{lem:cmng-dwn-inf} to find for some $
\lambda > 0 $
\begin{equs}
| u_{\ve}(t(\ve), x L_{\ve}) | \leqslant \frac{e^{t(\ve)}}{\sqrt{e^{2
t(\ve)} -1}} \leqslant 1 + \lambda e^{t(\ve)} \leqslant 1 + \lambda
\ve^{\frac{d}{2} - \alpha}.
\end{equs}
To establish the lower bound, consider for any constant $ K >0 $ and any $ \psi \in C^{2}_{b}$ the
following function (here $ \overline{\Phi} $ is as in \eqref{eqn:Phi-bar}), 
\begin{equs}
v_{\ve}(t, x) = \overline{\Phi} \Big(t, \psi(x) - \frac{K}{L_{\ve}} t
\Big).
\end{equs}
We see that, since $ \partial_{u} \overline{\Phi} \geqslant 0 $
\begin{equs}
\partial_{t} v_{\ve} &= L^{-2}_{\ve}\Delta v_{\ve} + v_{\ve}(1 - v_{\ve}^{2}) -
K L^{-1}_{\ve} \partial_{u} \overline{\Phi} - L^{-2}_{\ve}\partial_{u} \overline{\Phi} \Delta \psi -
L^{-2}_{\ve} \partial_{u}^{2} \overline{\Phi} | \nabla \psi |^{2} \\
& \leqslant L^{-2}_{\ve}\Delta v_{\ve} + v_{\ve}(1 - v_{\ve}^{2}) -
\frac{\partial_{u}\overline{\Phi} }{L^{2}_{\ve}} \Big(
K L_{\ve} - |\Delta \psi| - \frac{| \partial_{u}^{2} \overline{\Phi} |}{|
\partial_{u} \overline{\Phi} |} | \nabla \psi |^{2}
\Big).
\end{equs}
Now we observe that for $ t > 0 $, similarly to~\eqref{eqn:bounds-on-Phi} distinguishing the
cases $ |u| \leqslant e^{- t} $ and $ |u| \geqslant e^{-
t}$, we can bound
\begin{equ}
\bigg\vert \frac{\partial_{u}^{2} \overline{\Phi}(t, u)}{\partial_{u} \overline{\Phi} (t,
u)} \bigg\vert \lesssim \bigg\vert u \frac{ e^{2 t} -1}{ 1 +
u^{2}(e^{2 t} -1)} \bigg\vert \lesssim e^{t}.
\end{equ}
Hence there exists
a $ K (\psi) > 0 $ such that $ v_{\ve} $ is a subsolution to $
\partial_{t} u = L_{\ve}^{-2} \Delta u + u(1 - u^{2}) $ on the time interval $
[0, t(\ve)- t_{\star}(\ve) ] $ with initial condition $ \psi $. Here we use that $
\limsup_{\ve \to 0} \{t(\ve)- t^{\frac{1}{2}}_{\star}(\ve) \}  \leqslant 0$, so that for
a constant $ C>0 $ independent of $ \ve $ we have $ \exp( t(\ve)-
t_{\star}(\ve)) \leqslant C
L_{\ve} $. In particular, by our assumptions and by the upper bound of
\eqref{eqn:prf-slowmcf-1}, we can choose $ \psi \in C^{2}_{b} $ so that $
\psi(x) >0 $ for $ x $ in a closed ball $ \overline{B}_{\varrho}(x_{0}) $ about $ x_{0} $ and such that for some $ \ve(\psi) >0 $
\begin{equs}
u_{\ve}(t_{\star}(\ve), x L_{\ve}) \geqslant \overline{\Phi}( 0, \psi(x) ) =
\psi(x), \qquad \forall x \in
\RR^{d}, \ \ve \in (0, \ve(\psi)).
\end{equs}
Now, using that $ u =1 $ is an exponentially stable fixed point for $ \overline{\Phi} $, we have that for
every $ u > 0 $ there exists a $ \lambda(u) $ that can be chosen locally
uniformly over $ u $, such that $$ \overline{\Phi}(t(\ve) -
t_{\star}(\ve), u) \geqslant \overline{\Phi}(
 2\kappa \log{L_{\ve}} - C, u) \geqslant 1 - \frac{\lambda}{L_{\ve}^{2 \kappa}} $$ for all $ \ve $
sufficiently small, and for $ C>0 $ such that $ t (\ve) \geqslant
t^{\kappa}_{\star}(\ve) - C$ for all $ \ve $. Then by comparison, using that $
 \frac{t (\ve) - t_{\star}(\ve)}{L_{\ve}} \lesssim \frac{\log{L_{\ve}}}{L_{\ve}} \to 0 $, for $ \ve $ sufficiently small:
\begin{equs}
u_{\ve}(t(\ve), x L_{\ve}) \geqslant \inf_{y \in B_{\varrho}(x_{0})}  v_{\ve}(
t(\ve) - t_{\star}(\ve), y) \geqslant
1 - \frac{\lambda}{L_{\ve}^{2 \kappa}}, \qquad \forall x \in
B_{\varrho}(x_{0}) \;.
\end{equs}
This completes the proof of \eqref{eqn:prf-slowmcf-1} and of the lemma.
\end{proof}
The following proposition treats slightly longer time scales.

\begin{proposition}\label{prop:slower-than-mcf}
Consider $ (\Omega, \mF, \PP) $ as in Lemma~\ref{lem:prob-space} and fix any sequence $ \{ t (\ve) \}_{\ve
\in (0,1)}  $, with $ t(\ve) \geqslant t_\star(\ve) $, and such that
\begin{equs}
\liminf_{\ve \to 0} \big(t(\ve) - t^{\frac{1}{2}}_{\star}(\ve)\big) > 0 \; , \qquad \lim_{\ve \to 0}
\frac{t(\ve) - t_{\star}(\ve)}{T_{\ve}} = 0 \;.
\end{equs}
Then for any $ \delta, \zeta \in (0, 1) $
\begin{equ}
\lim_{\ve \to 0} \PP \( \| u_{\ve}(t(\ve), \cdot L_{\ve}) -
\sgn(\Psi(\cdot)) \|_{K_{\delta}} > \zeta\) =0 \; .
\end{equ}
\end{proposition}

\begin{proof}
As in the proof of the previous lemma, it suffices to prove that for any
subsequence $ \{ \ve_{n} \}_{n \in \NN}$ with $ \ve_{n} \in (0, 1),
\lim_{n \to \infty} \ve_{n} = 0  $ for which almost surely
\begin{equs}
u_{\ve}(t_{\star}(\ve), \cdot L_{\ve}) \to \Phi(0, \Psi(\cdot)) \ \ \text{
in } \ C_{\mathrm{loc}}(\RR^{d}) \; ,
\end{equs}
it holds that for all $ \delta \in (0, 1) $
\begin{equs}
\lim_{n \to \infty} \| u_{\ve_{n}}(t(\ve_{n}), \cdot L_{\ve_{n}}) -
\sgn(\Psi(\cdot)) \|_{K_{\delta}} = 0 \;.
\end{equs}
Hence we will work with a fixed realization of all random variables and for the
sake of clarity we will refrain from writing the subindex $ n $. In addition,
since the case $ \Psi < 0 $ is identical to the case $ \Psi>0 $, let us choose
an $ x_{0} \in \RR^{d}$ such that $ \Psi(x_{0}) > 0 $ and define $ s(\ve) =
t(\ve) - t_{\star}^{\frac{1}{2}} (\ve) $. We can also assume that $
s(\ve) \geqslant 0 $ for all $ \ve $. Our aim is to prove the convergence $\lim_{\ve \to 0}
u_{\ve}(t_{\star}^{\frac{1}{2} }(\ve) + s(\ve),x L_{\ve} ) =1 $ holds true for all $
x $ in a ball $ B_{\varrho}(x_{0}) $ about $ x_{0} $ of radius $ \varrho>0 $.

By Lemma~\ref{lem:front-formation-loglog} we already know that
there exist $ \overline{\lambda} (x_{0}), \varrho(x_{0}) >0$ and $\ve(x_{0}) \in (0, 1) $ such that:
\begin{equs}[eqn:slw-mcf-apr]
u_{\ve}(t_{\star}^{\frac{1}{2}} (\ve), x L_{\ve}) & \geqslant 1 - \frac{
\overline{\lambda}}{L_{\ve}},
\qquad & \forall \ve \in (0, \ve( x_{0})), \ \ x \in B_{\varrho}(x_{0}), \\
u_{\ve}(t_{\star}^{\frac{1}{2}} (\ve), x L_{\ve}) & \geqslant -1 -
 \overline{\lambda}  \ve^{\frac{d}{2} - \alpha}, \qquad & \forall x \in \RR^{d}.
\end{equs}
Here the second bound is a consequence of Lemma~\ref{lem:cmng-dwn-inf}
(in fact in the second statement $ \overline{\lambda} $ can be chosen deterministic and
independent of $ x_{0} $).
Our aim is to show that this front does not move after an additional time $
s(\ve) $.
Let us define $ \widetilde{u}_{\ve}(\sigma, x) =
u_{\ve}(t_{\star}^{\frac{1}{2}}(\ve) + \sigma s(\ve) , x  L_{\ve} ) $, which solves
\begin{equs}[eqn:prf-slow-mcf-ac]
\partial_{\sigma} \widetilde{u}_{\ve} = \frac{s(\ve)}{L^{2}_{\ve}} \Delta
\widetilde{u}_{\ve} + s(\ve) \widetilde{u}_{\ve}(1 -
\widetilde{u}_{\ve}^{2}), \qquad \widetilde{u}_{\ve}(0, \cdot) =
\widetilde{u}_{\ve, 0}(\cdot) \;,
\end{equs}
with an initial condition $ \widetilde{u}_{\ve, 0}(x) =
u_{\ve}(t^{\frac{1}{2}}_{\star}(\ve), x L_{\ve}) . $ Our purpose is to construct
an explicit subsolution $ \underline{u}_{\ve} $ to \eqref{eqn:prf-slow-mcf-ac}
with initial condition ``close'' to $ 1_{B_{\frac{\varrho}{2}}(x_{0})} $,
such that $ \lim_{\ve \to 0} \underline{u}_{\ve}(1, x) = 1$ for all $ x
$ in a neighbourhood of $ x_{0} $. Our ansatz is that close to the interface the
subsolution is of the following form,
for $ \zeta(\ve) = \frac{s(\ve)}{L_{\ve}^{2}} $:
\begin{equs}
\mf{q} \Big( L_{\ve} \ d(\zeta(\ve) \
\sigma, x) + L_{\ve} \mathcal{O}( \zeta(\ve))\Big) - \mathcal{O}(L_{\ve}^{-1}).
\end{equs}
Here $ d(\sigma, x) $ is the signed distance function associated to the mean
curvature flow evolution at time $ \sigma \geqslant 0 $ of the ball $
B_{\frac{\varrho}{2}}(x_{0}) $ with the sign convention $ d(0, x)> 0
$ if $ x \in B_{\frac{\varrho}{2}} (x_{0}) $, and $ d(0, x) \leqslant 0 $ if $
 x \in B_{\frac{\varrho}{2}}^{c}(x_{0})$  and $
\mf{q}(u) = \tanh (u) $ is the traveling wave solution to the Allen--Cahn
equation:
\begin{equs}
\ddot{\mf{q}} + \mf{q}(1 - \mf{q}^{2}) = 0, \qquad \lim_{x \to \pm \infty}
\mf{q} = \pm 1 \;.
\end{equs}

Our first step is to construct precisely the subsolution near the interface. Define $ \overline{d}(\sigma, x) = e^{- \mu \sigma} d (\sigma, x) $, for some
$ \mu >0 $ that will be chosen later on. We observe that there are $
\sigma^{\prime}> 0, \varrho^{\prime} \in (0, \varrho/2) $ such that $ d(\sigma,x) $ is smooth in the
set
\begin{equs}
Q_{\varrho} = [0, \sigma^{\prime}] \times \{ x\,:\,| x - x_{0} | \in [\varrho/2 -
\varrho^{\prime} , \varrho/2 + \varrho^{\prime} ] \}
\end{equs}
and (see for example \cite[Equation (6.4), p.663]{evans1991}) there exists a
constant $ C>0 $, depending on $ \varrho, \varrho^{\prime},
\sigma^{\prime} $, such that
\begin{equs}
\partial_{\sigma} d - \Delta d \leqslant C d \quad \text{ on } \ \ Q_{\varrho}.
\end{equs}
In particular, fixing $ \mu \geqslant C $ we have $ (\partial_{\sigma} - \Delta)
\overline{d} \leqslant 0 $ on $ Q_{\varrho} $.  Then consider, for some $ K_{1} > 0 $
\begin{equs}
w_{\ve}(\sigma, x) = \mf{q} \Big( L_{\ve} \ \overline{d} (\zeta(\ve) \sigma,
x) - K_{1} \frac{s(\ve)}{L_{\ve}} \sigma \Big)- \frac{\lambda}{L_{\ve}}.
\end{equs}
We claim that for $K_{1} $ sufficiently large $ w_{\ve} $ is a subsolution to
\begin{equs}[eqn:prf-slowmcf-2]
\partial_{\sigma} w_{\ve} - \zeta(\ve) \Delta w_{\ve}& - s(\ve) w_{\ve}(1 -
w_{\ve}^{2}) \leqslant 0, \quad \text{ on } \ Q_{\varrho}^{\ve},
\end{equs}
where
\begin{equs}
Q_{\varrho}^{\ve} = [0, \zeta(\ve)^{-1} \sigma^{\prime}] \times \{x\,:\, | x - x_{0}
| \in [\varrho/2 -
\varrho^{\prime} , \varrho/2 + \varrho^{\prime} ] \}.
\end{equs}
In fact, since $ \dot{\mf{q}} \geqslant 0 $, we can compute
\begin{equs}
\partial_{\sigma} w_{\ve} & = L_{\ve} \dot{\mf{q}} \ \zeta(\ve) \ \partial_{\sigma}
\overline{d} - K_{1}\frac{s(\ve)}{L_{\ve}} \dot{\mf{q}}\\
& =  \zeta(\ve) \Delta w_{\ve} + L_{\ve}  \zeta(\ve) \ \dot{\mf{q}} 
(\partial_{\sigma} - \Delta )\overline{d} - \zeta(\ve) L^{2}_{\ve} \ddot{\mf{q}}
| \nabla d |^{2}  - K_{1} \frac{s(\ve)}{L_{\ve}} \dot{\mf{q}}\\
& \leqslant \zeta(\ve) \Delta w_{\ve} - s(\ve) \ddot{\mf{q}}  - K_{1} \frac{s(\ve)}{L_{\ve}} \dot{\mf{q}}
\end{equs}
where we used that $ | \nabla d |^{2} =1 $ in a neighborhood of the boundary of $
B_{\frac{\varrho}{2}}$. Now we use the definition of $ \mf{q} $ to rewrite the
last term as
\begin{equs}
\zeta(\ve) \Delta w_{\ve} & + s(\ve)\mf{q}(1 - \mf{q}^{2})  -
K_{1}\frac{s(\ve)}{L_{\ve}} \dot{\mf{q}}.
\end{equs}
At this point we would like to replace $ \mf{q} $ with $ w_{\ve} =
\mf{q} - \frac{\lambda}{L_{\ve}} $ in the nonlinearity.
We observe that since $ u \mapsto u(1-u^{2}) $ is decreasing near $ u=1 $ and
$ u = -1 $ there exists a $ \gamma \in (0, 1) $ such that if $ | \mf{q} | \in (
\gamma, 1) $ and $ \ve $ is sufficiently small, then $ \mf{q}(1 - \mf{q}^{2})
\leqslant w_{\ve}(1 - w_{\ve}^{2}) $. On the other hand, on the set $ |\mf{q}|
\leqslant \gamma $ there exists a constant $ c(\gamma) >0 $ such that $
\dot{\mf{q}} \geqslant c(\gamma)>0 $. Hence in this last case:
\begin{equs}
\zeta(\ve) \Delta w_{\ve} & + s(\ve)\mf{q}(1 - \mf{q}^{2})  -
K_{1}\frac{s(\ve)}{L_{\ve}} \dot{\mf{q}} \\
& \leqslant   \zeta(\ve) \Delta w_{\ve}  + s(\ve) w_{\ve}(1 - w_{\ve}^{2}) \\
& \quad \quad + s(\ve) \bigg\{ \frac{\lambda}{L_{\ve}} - 3
\frac{\lambda}{L_{\ve}} w_{\ve}^{2} -  
\frac{c(\gamma) K_{1}}{L_{\ve}}  + \mathcal{O}\Big(\frac{1}{L^{2}_{\ve}}\Big)
\bigg\} 1_{\{ | \mf{q} | \leqslant \gamma \}}\\
& \leqslant \zeta(\ve) \Delta w_{\ve} + s(\ve) w_{\ve}(1 -
w_{\ve}^{2}),
\end{equs}
where the last inequality holds for all $ \ve $ sufficiently small, and
provided $ K_{1} $ is chosen large enough. Hence we have proven
\eqref{eqn:prf-slowmcf-2}.

The next step is to extend this subsolution $w_{\ve}$ to all $ x \in \RR^{d} $ (at
the moment it is defined only for 
$ | x - x_{0} | \in [\varrho/2 - \varrho^{\prime}, \varrho/2+ \varrho^{\prime}] $).
Here we follow two different arguments in the interior and the exterior of the
ball $ B_{\frac{\varrho}{2}} $. Let us start with the exterior. We observe that
for any fixed $ \varrho^{\prime
\prime} \in (0, \varrho^{\prime}) $ it holds that for all $ \ve $ sufficiently
small and some $ \lambda^{\prime} >0 $:
\begin{equs}
w_{\ve}  (\sigma, x) &\leqslant -1 - \frac{\lambda^{\prime}}{ L_{\ve}} , \\
& \forall (\sigma, x) \in [0,
\zeta(\ve)^{-1}  \sigma^{\prime}] \times \{ x  \ \colon \ | x - x_{0} | \in
[\varrho/2 + \varrho^{\prime \prime }, \varrho/2 + \varrho^{\prime}] \}.
\end{equs}
Here we use that asymptotically, for $ x \to - \infty $, we have $
\mf{q}(x) \leqslant -1 + 2 e^{- 2 x},$ so that by definition, for some $
c(x_{0}) > 0 $
\begin{equs}
\sup_{| x - x_{0} | \in [\varrho/2 + \varrho^{\prime \prime}  , \varrho/2+
\varrho^{\prime} ]}w_{\ve}(\sigma, x) & \leqslant  \mf{q} \left(-L_{\ve}\left( c(x_{0})
+ \mathcal{O}(\zeta(\ve) ) \right) \right) - \frac{\lambda}{L_{\ve}} \\
& \leq -1 + 2 \exp ( - \left( c(x_{0}) + \mO(\zeta_{\ve}) \right) L_{\ve}) - \frac{\lambda}{L_{\ve}} \leqslant
-1 - \frac{\lambda^{\prime}}{L_{\ve}},
\end{equs}
where the last inequality holds for $ \ve $ sufficiently small. Now consider $ \overline{\lambda} $ as in \eqref{eqn:slw-mcf-apr} and $
\overline{\Phi} $ as in \eqref{eqn:Phi-bar}. Then let $ S^{\ve} $ be
the set
\begin{equs}
S^{\ve} = \{   x  \ \colon \ | x - x_{0} | \in
[\varrho/2 + \varrho^{\prime \prime }, \varrho/2 + \varrho^{\prime}] \}.
\end{equs}
Since $ \sigma \mapsto
\overline{\Phi} (s(\ve) \sigma, -1 - \overline{\lambda} \ve^{\frac{d}{2} -
\alpha}) \eqdef \overline{\Phi}_{\ve}(\sigma)$
is a spatially homogeneous solution to \eqref{eqn:prf-slow-mcf-ac} we find
that for small $ \ve $,
\begin{equs}
\underline{w}_{\ve} (\sigma, x) = \begin{cases} \max \{ w_{\ve}(\sigma, x),
\overline{\Phi}_{\ve}(\sigma ) \}, \
& \text{ if } | x - x_{0} | \in [ \frac{\varrho}{ 2}  - \varrho^{\prime},
\frac{\varrho}{2}  + \varrho^{\prime}], \\
\overline{\Phi}_{\ve} (\sigma) \quad & \text{ else. } 
\end{cases}
\end{equs}
is a subsolution to \eqref{eqn:prf-slow-mcf-ac} in the viscosity sense (cf. the
proof of Lemma~\ref{lem:abs.v}) on the set
\begin{equs}
\ [0, \zeta(\ve)^{-1} \sigma^{\prime}] \times \{ x  \ \colon \ | x -
x_{0}| \geqslant \varrho/2 - \varrho^{\prime} \}.
\end{equs}
Here we use that asymptotically $ \lambda^{\prime} L_{\ve}^{-1} \gg \ve^{\frac{d}{2}
 - \alpha} $, so that $ \underline{w}_{\ve}(\sigma, x) =
\overline{\Phi}_{\ve} (\sigma) $ for all $ x \in S^{\ve} $.

Finally, we want to extend the subsolution to the interior of the ball $
B_{\frac{\varrho}{2}}(x_{0}) $. To complete this extension we consider a convex
combination between $ \underline{w}_{\ve} $ and the constant $ 1 -
\frac{\lambda}{L_{\ve}}.$ Let us fix a decreasing smooth
function $ \Upsilon \colon \RR \to
[0,1] $ such that $ \Upsilon (x) = 1 $ if $
x \leqslant  \big( \frac{\varrho - \varrho^{\prime}}{2} \big)^{2}  $, and $
\Upsilon (x) = 0 $ if $ x \geqslant
\big( \frac{\varrho - \varrho^{\prime}/2 }{2} \big)^{2}. $ Then define, for
some constant $ K_{2} > 0 $
\begin{equs}
\underline{u}_{\ve}(\sigma, x) = \big( 1 - & \Upsilon \big( | x- x_{0} |^{2}
+K_{2} \zeta(\ve) \sigma \big) \big) \underline{w}_{\ve}( \sigma , x) \\
& + \Upsilon( | x - x_{0} |^{2} + K_{2} \zeta(\ve) \sigma) \big(1 -
\frac{\lambda}{L_{\ve}} \big).
\end{equs}
Note that by considerations on the support of $ \Upsilon $ and the domain of
definition of $ \underline{w}_{\ve} $, $ \underline{u}_{\ve} $ is well defined.
We claim that if $ K_{2} $ is sufficiently large the function $
\underline{u}_{\ve} $ is a viscosity subsolution to \eqref{eqn:prf-slow-mcf-ac}
on $ [0, \zeta(\ve)^{-1}\sigma^{\prime}] \times \RR^{d} $. In fact, by our previous calculations we find
\begin{equs}
(\partial_{\sigma} - \zeta(\ve) \Delta) \underline{u}_{\ve} \leqslant
s(\ve) & \ (1-\Upsilon) \ \underline{w}_{\ve}(1 - \underline{w}_{\ve}^{2})  \\
& + \zeta(\ve) \Big( \underline{w}_{\ve} - \Big(1 -
\frac{\lambda}{L_{\ve}} \Big) \Big) \Big[ 2 \dot{\Upsilon} +
4 \ddot{\Upsilon} | x - x_{0} |^{2} - K_{2} \dot{\Upsilon} \Big] \\
 \leqslant s(\ve) & \ (1-\Upsilon) \ \underline{w}_{\ve}(1 -
\underline{w}_{\ve}^{2}) \\
& + 4 \ \zeta(\ve) \ \ddot{\Upsilon} \ | x - x_{0} |^{2}( \underline{w}_{\ve} - (1 -
\lambda/ L_{\ve}) ),
\end{equs}
assuming $ K_{2} \geqslant 2 $, and using that $ \underline{w}_{\ve} \leqslant 1 -
\frac{\lambda}{L_{\ve}}$ and $ \dot{\Upsilon} \leqslant 0 $, as $ \Upsilon $ is
decreasing.
Now on the set $ \{ \Upsilon >0 \}  $, we have $ \underline{w}_{\ve} \geqslant
0$, provided $ \ve $ is sufficiently small. So using the concavity of $ [0, 1] \ni u \mapsto u
(1- u^{2}) $:
\begin{equs}
(1-\Upsilon) \underline{w}_{\ve}(1 - \underline{w}_{\ve}^{2}) & \leqslant \underline{u}_{\ve}(1 -
\underline{u}_{\ve}^{2}) - \Upsilon (1 - \lambda/L_{\ve})(1 - ( 1 -
\lambda/L_{\ve})^{2}).
\end{equs}
Furthermore, we can find a constant $ \nu \in (0, 1) $ such that
\begin{equs}
\ddot{\Upsilon} \geqslant  0 \quad \text{ if } \ \ \Upsilon \leqslant \nu.
\end{equs}
Hence we see that on the set $ \{ \Upsilon \leqslant \nu \}$
\begin{equs}
(\partial_{\sigma} - \zeta(\ve) \Delta) \underline{u}_{\ve}& \leqslant
s(\ve)  \ (1-\Upsilon) \ \underline{w}_{\ve}(1 - \underline{w}_{\ve}^{2})  \\
& \leqslant s(\ve)\underline{u}_{\ve} (1 - \underline{u}_{\ve}^{2}).
\end{equs}
On the other hand, on the set $ \{ \Upsilon> \nu \}$ we have for some $ C>0 $:
\begin{equs}
(\partial_{\sigma} - \zeta(\ve) \Delta) \underline{u}_{\ve} & \leqslant
s(\ve)\underline{u}_{\ve}  (1 - \underline{u}_{\ve}^{2}) - \nu s(\ve) (1 - \lambda/L_{\ve})(1 - ( 1 -
\lambda/L_{\ve})^{2})+  C \ \zeta(\ve) \ | \ddot{\Upsilon}| \\
& \leqslant s(\ve)\underline{u}_{\ve}  (1 - \underline{u}_{\ve}^{2}) - 2 \nu
\frac{\lambda s(\ve)}{L_{\ve}}+ \mathcal{O}(\zeta(\ve))  \\
& \leqslant  s(\ve)\underline{u}_{\ve}  (1 - \underline{u}_{\ve}^{2}),
\end{equs}
where the last inequality holds for $ \ve $ is sufficiently small.

To conclude, for $ \ve $ sufficiently small, we have constructed a subsolution
to \eqref{eqn:prf-slow-mcf-ac} such that, by \eqref{eqn:slw-mcf-apr} and up to
choosing $ \lambda $ sufficiently large, the initial condition satisfies
\begin{equ}
\widetilde{u}_{\ve, 0}(\cdot) \geqslant \underline{u}_{\ve}(0, \cdot).
\end{equ}
By comparison, since $ \lim_{\ve \to 0}
\underline{u}_{\ve}(1, x) = 1$, for all $ x  $ in a neighborhood of $
x_{0} $ and through the upper bound of Lemma~\ref{lem:cmng-dwn-inf} our proof is complete.
\end{proof}
The next result establishes convergence to level set solutions of mean
curvature flow. Recall that we have defined $$ U_{\ve} (\sigma, x) = u_{\ve}( \sigma
T_{\ve} + \tau_{\star}(\ve), x L_{\ve}) = u_{\ve}( t_{\star}(\ve) + (\sigma-1)
T_{\ve}, x L_{\ve}).$$
In these variables the initial condition for the mean curvature flow appears at time $ \sigma = 1 $.
The level set formulation of mean curvature
flow is then given by viscosity solutions to the following equation for $
(t, x) \in [1, \infty) \times \RR^{d} $
\begin{equs}[eqn:lvl-mcf]
\partial_{t} w= \Delta w - \frac{(\nabla w)^{\otimes 2} : \nabla^{2}w}{| \nabla w
|^{2}}  , \qquad w(1, x) = w_{1}(x).
\end{equs}
Here $ (\nabla w)^{\otimes 2} : \nabla^{2}w = \sum_{i,j=1}^{d}
\partial_{i} w \partial_{j}w \partial_{ij}w $.
If $ w_{1} $ is uniformly continuous on $ \RR^{d} $ there exists a unique viscosity solution
to the above equation, see e.g. \cite[Theorem 1.8]{IshiiSouganidis1995}. Furthermore, we will be only interested in the evolution of the sets $ \{ w > 0 \}, \{ w < 0 \} $
and $ \{ w = 0 \} $, which motivates the following definition.

\begin{definition}\label{def:mcf}
For any $ f \in C_{\mathrm{loc}}(\RR^{d} ; \RR) $ we define $ v(f ; \cdot,
\cdot ) \colon [1, \infty) \times \RR^{d} \to \{ -1, 0, 1 \} $ by
\begin{equs}
v(f; \sigma, x) = \sgn \big( w(\sigma, x) \big), \qquad (\sigma, x) \in [1,
\infty) \times \RR^{d},
\end{equs}
with $ w $ the viscosity solution to
\eqref{eqn:lvl-mcf} with an arbitrary initial condition
$ w_{1} \in C^{2}_{b}(\RR^{d}; \RR) $ satisfying:
\begin{equs}
\{ w_{1} > 0 \} = \{ f >0 \}, \quad \{ w_{1} <0 \} = \{ f < 0 \}, \quad \{ w_{1} = 0 \} = \{ f = 0 \}.
\end{equs}
The function $ v(f ; \cdot, \cdot) $ does not depend on the particular choice
of $ w_{1} $ by \cite[Theorem 5.1]{evans1991}.
\end{definition}
We recall the definition fo the sets $ K_{\delta} $ as in \eqref{eqn:K-sets}
for $ \delta \in (0, 1) $:
\begin{equ}
K_{\delta}  = \{ z = (\sigma, x) \in (1, \infty) \times \RR^{d}  \ \colon  \ | z | \leqslant
\delta^{-1}, \ \sigma > 1 + \delta,  \ d(z, \Gamma) \geqslant \delta\} \;.
\end{equ}
Now we can state our concluding result.
\begin{proposition}\label{prop:conv-mcf}
Consider $ (\Omega, \mF, \PP) $ as in Lemma~\ref{lem:prob-space} and let $
v(\Psi ; \cdot, \cdot) \colon [1,
\infty) \times \RR^{d} \to \{ -1, 0, 1 \} $ be as in Definition~\ref{def:mcf}.
Then for any $ \delta, \zeta \in (0, 1) $
\begin{equs}
\lim_{ \ve \to 0} \PP \( \| U_{\ve}(\cdot) - v(\Psi; \cdot) \|_{K_{\delta}}
\geqslant \zeta  \) =0 \;.
 \end{equs}
\end{proposition}

\begin{proof}

In analogy to the previous results, up to considering suitable subsequences we
can assume that $ U_{\ve}(1, \cdot) \to \overline{\Phi}(0, \Psi (\cdot)) $
almost surely in $ C_{\mathrm{loc}}(\RR^{d}) $ as $ \ve \to 0 $. Then it
suffices to prove that almost surely, for any $ \delta \in (0, 1) $
\begin{equs}
 \lim_{\ve \to 0} \| U_{\ve}(\cdot) - v(\Psi; \cdot) \|_{K_{\delta}} = 0 \;.
\end{equs}

In the setting just introduced, our aim is to construct suitable super- and sub-solutions to $
u_{\ve}(t_{\star}(\ve) + (\sigma-1) T_{\ve}, x L_{\ve}) $ with initial conditions that
are independent of $ \ve $ and constant outside of a compact set. After establishing convergence to mean curvature
flow for these super- and sub-solutions we will use comparison to obtain the
convergence of the original sequence. For convenience we will restrict
ourselves to the construction of subsolutions, which will guarantee convergence on the set $
\{ v( \Psi; \cdot, \cdot) > 0 \} $. The construction of supersolutions is
analogous.

  Consider a smooth function $ \varphi \colon \RR^{d} \to [-1, 0] $ such that $
  \varphi(x) =0$, for all $| x | \leqslant
  \frac{1}{2} $ and $ \varphi(x) = -1,$ for all $ | x | \geqslant 1$. Then for
any $ R \geqslant 1 $ define $ \varphi_{R}(x) =
  \varphi(x/R) $ and consider
  \begin{align*}
    \underline{u}_{R} & = \varphi_{R} + (1 +
    \varphi_{R}) \frac{\Psi}{\sqrt{1 +  \Psi^{2}}} - R^{-1}.
  \end{align*}
In particular by our assumptions for any $ R \geqslant 1$ there
exists a $ \ve(R) \in
(0, 1) $ such that
\begin{equs}[eqn:super-sub-solutions]
\underline{u}_{R}(x) \leqslant U_{\ve}(1, x), \quad \forall x \in \RR^{d}, \quad \ve \in (0,
\ve(R)).
\end{equs}
Here we use the convergence $ U_{\ve}(1, \cdot) \to \overline{\Phi}(0, \Psi
(\cdot)) $ in $ C_{\mathrm{loc}} (\RR^{d}) $ together with the a-priori bound
from Lemma~\ref{lem:cmng-dwn-inf}, which guarantees that $ U_{\ve}(\sigma,
\cdot) \geqslant -1 - R^{-1}$ for $ \sigma \geqslant  1 $ and $ \ve $
sufficiently small.
Moreover, locally uniformly over $ x \in \RR^{d} $:
$\lim_{R \to \infty} \underline{u}_{R}(x) =
\overline{\Phi}(0, \Psi(x))$.
Now, let $\underline{u}_{R, \ve} (\sigma,x) $ be the solution
to the rescaled Allen--Cahn equation
\begin{equ}[eqn:defn-super-sub-eqn]
\partial_{t} \underline{u}_{R, \ve}  = \Delta \underline{u}_{R,
\ve} +
T_{\ve} \underline{u}_{R, \ve}(1 - \underline{u}_{R,
\ve}^{2}), \qquad \underline{u}_{R, \ve}(1, x) =
\underline{u}_{R}(x).
\end{equ}
Then by \cite[Theorem 4.1]{barles1998}
$\lim_{\ve \to 0} \underline{u}_{R , \ve}  = \sgn(
\underline{w}_{R}),$ in $C_{\mathrm{loc}}((1, \infty )
\times \RR^{d} \setminus \{ \underline{w}_{R} = 0 \})$,
where $ \underline{w}_{R} $ is the unique level set solution to
\eqref{eqn:lvl-mcf} with initial condition $ \underline{u}_{R} $.
By comparison, using the ordering \eqref{eqn:super-sub-solutions} and the fact
that $ U_{\ve} $ also solves $ \eqref{eqn:defn-super-sub-eqn} $ with a
different initial condition, we obtain that for any $ R \geqslant 1$
  \begin{align*}
    \lim_{\ve \to 0} U_{\ve} (\sigma, x ) = 1& \qquad \text{
locally uniformly over } \qquad (\sigma, x) \in \underline{\mP}^{R},
  \end{align*}
  where $ \underline{\mP}^{R} = \{ \underline{w}_{R}( \cdot, \cdot) >
  0\} \subseteq (1, \infty) \times \RR^{d} $.

  Now we would like to pass similarly to the limit $ R \to \infty $, but we
have to take care of the fact that the limiting initial condition $
\lim_{R \to \infty} \underline{u}_{R} = \Psi / \sqrt{1 + \Psi^{2}} $ may not be uniformly
  continuous on $ \RR^{d} $, which complicates the construction of solutions.
Of course, this problem is not significant, since we can rescale the initial
condition arbitrarily by multiplying with a positive function without modifying
the mean curvature flow evolution. 

Hence consider a partition of the unity $ \{ \varphi_{k} \}_{k \in \NN}$ of $
\RR^{d} $. Namely, for every $ k \in \NN $, let $ \varphi_{k}
\colon \RR^{d} \to [0, \infty) $ be a smooth function with compact support,
such that $ \sum_{k \in \NN} \varphi_{k}(x) = 1.$ We will assume that the
partition is locally finite, in the sense that there exists an $ M \in \NN $
such that for every $ k_{1}, \ldots, k_{M} \in \NN $ with $ k_{i} \neq
k_{j} $ for all $ i \neq j $ we have $ \cap_{k =1}^{M}
\supp(\varphi_{k}) = \emptyset $. In addition we assume that $ \sup_{k \in \NN}
\| \varphi_{k} \|_{C^{2}} < \infty $. Then let 
\[ a_{k} = 1 + \sup_{x \in \supp(\varphi_{k}) } \big\{ | \Psi (x) | + | \nabla \Psi (x)| + |
\nabla^{2} \Psi(x)| \big\}\]
and define
\begin{align*}
  g(x) = \sum_{k
  \in \NN} \frac{1}{a_{k}} \varphi_{k}(x) \sqrt{1 + \Psi^{2}(x)}.
\end{align*}
Now we can consider the solution $ \underline{v}_{R}$ to
  \eqref{eqn:lvl-mcf} with initial condition $
  \underline{v}_{R, 1}(x) = g (x) \cdot \underline{u}_{R, 1}
  (x)$, so that from the definition of $ g $ we obtain:
\begin{equation}\label{eqn:prf-mcf-uniform-bound}
  \sup_{R \geqslant 1}   \| \underline{v}_{R , 1} \|_{C^{2}_{b}(\RR^{d})}  < \infty.
\end{equation}
  By \cite[Theorem 5.1]{evans1991} the evolution of the interface does not depend on the
  particular choice of the initial condition, as long as they share the same
  initial interface. In particular, since $ g $ is strictly positive we find
  that
  \[ \underline{\mP}^{R} = \{ \underline{v}_{R}(\cdot, \cdot) > 0 \}.\] 
The
bound \eqref{eqn:prf-mcf-uniform-bound} now guarantees a uniform bound
on the solutions (this is the consequence of the comparison principle in \cite[Theorem 4.1]{evans1991}) for any $ \alpha \in (0, 1) $:
\begin{equ}
  \sup_{R \geqslant 1} \| \underline{v}_{R} \|_{C^{\alpha}_{b}((1, \infty) \times
\RR^{d})} < \infty.
\end{equ}
Since in addition 
$$ \lim_{R \to \infty}\underline{v}_{R, 1} = \sum_{k \in \NN} \varphi_{k}
\frac{\Psi}{a_{k}} \eqdef w_{1},$$ 
with $ w_{1} \in C^{2}_{b}(\RR^{d}; \RR) $, using compactness as well as stability of viscosity solutions we find that
$\lim_{R \to \infty} \underline{v}_{R}(\cdot, \cdot) = w(\cdot, \cdot )$ in  
$C([1, \infty) \times \RR^{d}; \RR)$,
where the latter is the unique viscosity solution to \eqref{eqn:lvl-mcf} with
initial condition $ w_{1} $. This completes the proof.
\end{proof}
The following corollary relates level set solutions to mean curvature flow to
classical solutions in the case $ d=2 $.
Consider the unit torus $ \TT = \RR/\ZZ $. A continuous (resp.\
smooth) closed curve is any continuous (resp.\ smooth) map $ \gamma \colon \TT \to
\RR^{2} $. We say that the curve is non self-intersecting if the map $ \gamma
$ is injective.

\begin{corollary}\label{cor:2D-mcf}
If $ d=2 $, the set $ \Gamma_{1} = \{ \Psi = 0 \} $ consists of a countable
collection of disjoint smooth, closed, non self-intersecting curves: $ \Gamma_{1} = \bigcup_{i \in \NN}
\gamma_{i}(\TT). $
Then, for any $ \sigma \geqslant 1 $ the set $ \Gamma_{\sigma} = \{ v( \Psi;
\sigma, \cdot) = 0 \}$, with $ v( \Psi) $ as in
Proposition~\ref{prop:conv-mcf}, is the disjoint union $ \Gamma_{\sigma} =
\bigcup_{i \in \NN} \gamma_{i, \sigma}(\TT) $ of continuous curves $
\gamma_{i, \sigma} $ which are the mean curvature flow evolution of
the curves $ \gamma_{i} $, in the sense of \cite{Grayson87}.
\end{corollary}

\begin{proof}
The fact that $
\Gamma_{1} $ is the disjoint union of smooth closed curves is proven in
Lemma~\ref{lem:properties-field}. Each of these curves evolves according to the
level set notion of mean curvature flow, which for concreteness we denote by $
\mM_{\sigma} \gamma_{i}(\TT) $, for $ i \in \NN $ and $ \sigma \geqslant 1 $.
We observe that \cite[Theorem 7.3]{evans1991} implies that the sets $ \mM_{\sigma}
\gamma_{i}(\TT) $ are pairwise disjoint over $ i \in \NN $, for any $ \sigma >1
$. In addition the level set mean curvature flow evolution of a smooth closed curve coincides
with the classical evolution by \cite[Theorem 6.1]{evans1991}, as long as the
latter is defined. Since by Lemma~\ref{lem:properties-field} our curves are non self-intersecting, \cite{Grayson87}
proves that the classical evolution is defined for all times, and our result
follows.
\end{proof}

\begin{lemma}\label{lem:properties-field}
Let $ \Psi $ be as in~\eqref{eqn:def-Psi}, in dimension $ d=2
$, and consider the random
set $\Gamma_{1}  = \{ x  \ \colon \ \Psi(x) = 0 \} \subseteq \RR^{2}.$
Then $ \PP - $almost surely the set $ \Gamma_{1} $ is a
countable union of smooth, non self-intersecting and disjoint closed curves:
$\Gamma_{1} =
\bigcup_{i \in \NN} \gamma_{i}(\TT),$ with $ \gamma_{i} \colon \TT \to
\RR^{2} $ smooth.

\end{lemma}

\begin{proof}
By \cite[Corollary 1.4]{beffara2017} there exists a null set $
\mf{N}_{1} $ such that for all $ \omega \in \mf{N}_{1}^{c} $ every connected
component of $ \Gamma_{1}(\omega), \ \mP_{1}(\omega) $ and $ \mN_{1}(\omega) $ is bounded.

Hence the proof is complete if we show that the boundary of every region is
given by a smooth closed curve. Let us first show that there exists a null set $
\mf{N}_{2} $ such that for all $ \omega \in \mf{N}_{2}^{c} $ it holds that
\begin{equ}[eqn:non-zero-grad]
| \nabla \Psi (\omega, x)| \neq 0, \qquad \forall x \in \Gamma_{1}(\omega).
\end{equ}
This follows from Bulinskaya's lemma, see for example \cite[Lemma 11.2.10]{adler2009}, as long as we
can prove the following nondegeneracy condition, namely that the map
\begin{equ}
 h \colon \Omega \times \RR^{2} \to \RR^{3}, \qquad h(\omega, x) =
(\nabla \Psi (\omega, x), \Psi(\omega, x)),
\end{equ}
satisfies that for any $ x \in \RR^{2} $ the probability measure
\begin{equ}
\PP_{x}(A) = \PP( h(x) \in A) 
\end{equ}
has a density $ p_{x}(y) $ with respect to Lebesgue measure on $
\RR^{3} $ such that for some $ C>0$:
$| p_{x}(y) | \leqslant C, \ \forall x \in \RR^{2}, \ y \in \RR^{3}.$ In our
case this condition is trivially satisfied as for every $ x \in \RR^{2} $, $
h (x) $ is distributed as a  Gaussian vector in $ \RR^{3}$ with diagonal
covariance matrix (here $ \lambda = \frac{d - 2 \alpha}{2} $ and $ C(\lambda)>0
$ is a constant):
\begin{equs}
\EE [ \partial_{i} \Psi (x) \partial_{j} \Psi (x)] & = \frac{1}{
(4 \pi \lambda)^{d} (2 \lambda)^{2}}  \int_{\RR^{2}} 
e^{ - 2 \frac{| y |^{2}}{4 \lambda}} y_{i} y_{j} \ud y = C(\lambda) 1_{\{i =
j\}}, \\
\EE [ \partial_{i} \Psi (x)\Psi (x)] & = \frac{1}{
(4 \pi \lambda)^{d} 2 \lambda}  \int_{\RR^{2}} 
e^{ - 2 \frac{| y |^{2}}{4 \lambda}} y_{i}  \ud y = 0.
\end{equs}
It follows that \eqref{eqn:non-zero-grad} holds true. An application of the implicit function
theorem allows us to deduce our result.
\end{proof}

\bibliographystyle{Martin}
\bibliography{refs}

\end{document}